\documentclass[11pt]{amsart}

\usepackage{amsmath}
\usepackage{amssymb,amsfonts,mathrsfs, amsthm, stmaryrd}
\usepackage{mathtools}
\usepackage[colorlinks=true,linkcolor=blue,citecolor=blue]{hyperref}
\usepackage[margin=1 in,heightrounded=true,centering]{geometry}
\usepackage[applemac]{inputenc}

\usepackage{amscd}
\usepackage{verbatim}
\usepackage{euscript}

\newtheorem{theorem}{Theorem}[section]
\newtheorem*{theorem-nonum}{Theorem}
\newtheorem{prop}[theorem]{Proposition}

\newtheorem{lemma}[theorem]{Lemma}

\newtheorem{cor}[theorem]{Corollary}
\newtheorem{definition}[theorem]{Definition}

\newtheorem{maintheorem}{Theorem}

\numberwithin{equation}{section}

\newcommand{\vep}{\varepsilon}
\renewcommand{\epsilon}{\vep}

\newcommand{\tr}{ \mathrm{tr} }

\newcommand{\riem}{\mathrm{R}}
\newcommand{\Rm}{\mathrm{Rm}}
\newcommand{\Rc}{\mathrm{Rc}}

\newcommand{\inn}[3]{\left\langle #1, #2 \right\rangle_{#3}}

\DeclareFontFamily{U}{MnSymbolC}{}
\DeclareSymbolFont{MnSyC}{U}{MnSymbolC}{m}{n}
\DeclareFontShape{U}{MnSymbolC}{m}{n}{
    <-6>  MnSymbolC5
   <6-7>  MnSymbolC6
   <7-8>  MnSymbolC7
   <8-9>  MnSymbolC8
   <9-10> MnSymbolC9
  <10-12> MnSymbolC10
  <12->   MnSymbolC12}{}
\DeclareMathSymbol{\intprod}{\mathbin}{MnSyC}{'270}

\newcommand{\bS}{\mathbb{S}}
\newcommand{\bR}{\mathbb{R}}
\newcommand{\bN}{\mathbb{N}}

\newcommand{\bH}{\mathbb{H}}

\newcommand{\bB}{\mathbb{B}}

\newcommand{\Mbar}{\overline{M}}
\newcommand{\Mdbar}{\widehat{M}} % the manifold with corners
\renewcommand{\hbar}{\overline{h}}

\newcommand{\sQ}{\mathscr{Q}}
\newcommand{\sC}{\mathcal{C}}
\newcommand{\sG}{\mathcal{G}}
\newcommand{\sB}{\mathscr{B}}

\newcommand{\gbar}{\overline{g}}
\newcommand{\ghat}{\hat{g}}
\newcommand{\hhat}{\hat{h}}
\newcommand{\qhat}{\hat{q}}
\newcommand{\qbar}{\overline{q}}

\newcommand{\cF}{\mathcal{F}}
\newcommand{\cI}{\mathcal{I}}
\newcommand{\cV}{\mathcal{V}}
\newcommand{\cU}{\mathcal{U}}
\newcommand{\pa}{\partial}
\newcommand{\CI}{\mathcal{C}^{\infty}}
\usepackage{color}

\newcommand\cf{cf\@. }

\begin{document}

\title{Geometrically finite Poincar\'e-Einstein metrics}

\begin{abstract}
We construct new examples of Einstein metrics by perturbing the conformal infinity of geometrically finite hyperbolic metrics and by applying the inverse function theorem in suitable weighted H\"older spaces.
\end{abstract}

\author{Eric Bahuaud}
\address{Department of Mathematics, Seattle University}
\email{bahuaude(at)seattleu.edu}
\author{Fr\'ed\'eric Rochon}
\address{D\'epartement des math\'ematiques, UQ\`AM}
\email{rochon.frederic(at)uqam.ca}

\date{\today}

\maketitle
\section{Introduction}

Let $\Mbar$ be a smooth compact $n$-manifold with boundary $\partial M$ and interior $M$.  Let $\rho \in C^{\infty}(\Mbar)$ be a non-negative function vanishing to first order precisely on $\partial M$.  A Riemannian metric $g$ on $M$ is $C^{k,\alpha}$ conformally compact if $\rho^2 g$ extends to a $C^{k,\alpha}$ metric on $\Mbar$.  The induced conformal class of metrics on $\partial M$, $[(\rho^2 g)|_{T\partial M}]$ is the conformal infinity of $g$.

\

The Poincar\'e metric on the unit ball $\bH^n := \bB^{n} = \{x\in \bR^n: |x| < 1\}$ given by $h=\frac{4}{(1-|x|^2)^2} \sum (dx^i)^2$ is an example of a conformally compact Einstein metric.  Here $\rho = (1-|x|^2)/2$, and the conformal infinity is the conformal class of the round metric $\hat{h}$ on $\bS^{n-1}$. In 1991, via the inverse function theorem, Robin Graham and John M. Lee proved the existence of asymptotically hyperbolic Einstein metrics on the unit ball with prescribed conformal infinity sufficiently close to the Poincar\'e model \cite{GrahamLee}.  
Such metrics are variously called Poincar\'e-Einstein, or asymptotically hyperbolic Einstein or conformally compact Einstein metrics.  

\

This result was subsequently generalized by Biquard \cite{Biquard} in complex, quarternionic and octonionic hyperbolic spaces through deformations of the appropriate notion of conformal infinity.  In another direction, Lee \cite{Lee} extended the perturbation result of \cite{GrahamLee} by deforming  the conformal infinity of more general Poincaré-Einstein metrics.  To be able to apply the inverse function theorem, Lee needed the $L^2$-kernel of the linearized problem to vanish, a condition automatically satisfied when the Poincaré-Einstein metric has nonpositive sectional curvature, for instance when it is a convex co-compact hyperbolic metric.     More recently, Albin, in \cite{Albin}, generalizes \cite{GrahamLee}, still by deforming the conformal infinity of the hyperbolic space, but replacing the Einstein equation by an equation involving a linear combination of Lovelock tensors, obtaining in this way examples of Poincaré-Lovelock metrics.  In a very different direction, let us point out also that Enciso and Kamran in \cite{EK2019} have obtained a Lorentzian version of \cite{GrahamLee, Lee}.   

\

More generally, instead of convex co-compact hyperbolic metrics, one can consider  geometrically finite hyperbolic metrics.  As observed by Mazzeo and Philips \cite{MazzeoPhillips}, these metrics admit a natural compactification to a manifold with corners with at most codimension 2 corners.  When the volume is infinite, there is a special boundary hypersurface,  $H_0$, corresponding to the directions of maximal volume growth, while the other boundary hypersurfaces are mutually disjoint and correspond to cusps.  A cusp of maximal rank yields an end of finite volume and the corresponding boundary hypersurface $H_i$ is a closed submanifold disjoint from all the other boundary hypersurfaces.  If instead the cusp is of intermediate rank, then the corresponding boundary hypersurface $H_i$ is a manifold with boundary with non-empty intersection with $H_0$, so that $\pa H_i = H_i\cap H_0$.  When the metric is convex co-compact, $H_0$ is the only boundary hypersurface and we recover the usual compactification to a manifold with boundary.  

\

Coming back to a general geometrically finite hyperbolic metric $g$, this suggests that if $\rho$ now denotes a boundary defining function of $H_0$, then 
$\rho^2g|_{TH_0}$ should be a representative of the conformal infinity of $g$.  As explained in \S~\ref{0F.0}-\ref{s:gfhyp}, the representative $\rho^2g|_{TH_0}$ is naturally a complete metric on $H_0\setminus \pa H_0$ with foliated cusps at infinity in the sense of \cite{R2012}.   Deforming the conformal class of $\rho^2g|_{TH_0}$, one could therefore hope to obtain a corresponding Einstein deformation of the metric $g$, at least for small enough deformations. To avoid certain complications near the corners, one can even restrict to deformations of $\rho^2g|_{TH_0}$ compactly supported in an open set $U$ of $H_0$ with closure disjoint form $\pa H_0$, since near such an open set, the metric $g$ locally takes the form of a conformally compact metric.  The main purpose of the present paper is to establish that indeed, such Einstein deformations are possible, yielding a new class of complete Einstein metrics that we call  geometrically finite Poincaré-Einstein metrics.  The precise result is as follows (see Theorem~\ref{mt.1} below for a statement with more details on the regularity of the metric).

\begin{maintheorem} \label{thm:main}
For $n \geq 4$, and $k > n$, let $(M^n,h)$ be a geometrically finite hyperbolic metric with $n_c$ intermediate rank cusps and no cusps of maximal rank.  If $n=4$, suppose moreover that each cusp is of rank $1$.  Then there exists $\vep > 0$ such that for any smooth compactly supported symmetric $2$-tensor $\qhat\in C^{\infty}_c(U;T^*U\otimes T^*U)$ and perturbation $\hat{g} = \hat{h} + \qhat$ with $\|\qhat\|_{C^{k,\alpha}(U)} < \vep$, there is a geometrically finite asymptotically hyperbolic metric with cusps $g$ on $M$ such that 
$\rho^2 g$ is continuous on $U$ with $(\rho^2 g)|_{TU} = \hat{g}$ and  $g$ is Einstein, i.e. 
\[ \Rc(g) + (n-1)g = 0. \]
\end{maintheorem}

\

 We now discuss the key aspects of the proof in a high-level manner.  As in \cite{GrahamLee}, we replace the Einstein equation above by an appropriate gauge-adjusted equation.  In particular we define an operator $Q$ that takes pairs of metrics $(g,t)$ to symmetric two-tensors by
\begin{equation} \label{definition-Q}
 Q(g,t) = \Rc(g) + (n-1)g - \delta_g^* (g t^{-1} (\delta_g (G_g t))), 
\end{equation}
where $(G_g t)_{ij} := t_{ij} - \frac{1}{2} g^{pq} t_{pq} g_{ij}$, $\delta_g$ is the divergence of a symmetric $2$-tensor and $\delta_g^*$ is its formal adjoint.  That it is enough to solve $Q(g,t) = 0$ for sufficiently regular $g$ and $t$ will follow from the maximum principle.  

\

It will be important to obtain approximate solutions to $Q(g,t) = 0$.  This is the step where our assumption that the conformal deformation is compactly supported yields important simplifications.  Indeed, the argument of Graham-Lee to construct an approximate solution is local in the directions tangent to the boundary, so can be applied directly to our setting as long as the conformal perturbations are compactly supported.  

\

Having obtained a sufficiently good asymptotic solution, the result follows by applying the inverse function theorem.  This requires an isomorphism theorem for an appropriate linear problem related to the linearization of $Q$ at the hyperbolic metric.  Writing $r = uh + r_0$ where $u$ is a smooth function and $r_0$ is trace-free with respect to $h$, we will see the existence problem for a perturbative gauge-adjusted Einstein metric reduces to the invertibility on weighted tensor field spaces of the operator
\begin{equation} \label{linearize-scrQ}
 L r = \frac{1}{2} ( (\Delta^h + 2(n-1)) (uh) + (\Delta^h - 2) r_0 ).
\end{equation}
Comparing to \cite{GrahamLee, Lee}, this is the step where a new approach is required.  This starts with the definition of the appropriate weighted Hölder spaces, where, instead of patching together descriptions in local coordinate charts, we were led to take the more global point of view of Melrose \cite[\S~2.5]{MelroseMWC} and Ammann-Lauter-Nistor \cite{ALN2004} relying on the compactification of Mazzeo-Phillips \cite{MazzeoPhillips}.  

\

Given these weighted Hölder spaces, the isomorphism theorem is proved using a suitable exhaustion of the geometrically finite hyperbolic manifold by bounded domains with smooth boundaries.  On these domains, we rely on the argument of Koiso as explained in \cite[Proof of Theorem~A]{Lee} to show that the linearized problem with Dirichlet boundary conditions is an isomorphism.  To obtain an isomorphism for the whole manifold, one needs though to have uniform Schauder estimates on these domains, that is, an analogue of the basic estimate of \cite{GrahamLee}.  As in \cite{GrahamLee}, we can still use a barrier function to get good control outside some large compact set, which helps in determining weights for which we will have an isomorphism theorem.  In particular, when there are maximal rank cusps, we can only get isomorphism theorems for unbounded weight functions in the cusp ends of finite volume, making them unsuitable for a possible application of the inverse function theorem.  This is the reason why maximal rank cusps are excluded in Theorem~\ref{thm:main}, as well as cusps of rank 2 when $n=4$.   

\

Unfortunately however, the part of the argument of \cite{GrahamLee} giving uniform control on domains is very particular to the model hyperbolic space and does not apply to our setting.   On the other hand, the approach of Lee \cite{Lee}, which passes through a careful study of Fredholm properties of geometric operators (see also the work of Mazzeo and Melrose \cite{MM1987,Mazzeo88,MazzeoEdge} on that matter), could possibly be adapted to our setting, but would definitely require a substantial amount of work.   We found instead a more direct route to obtain the isomorphism theorem.  Our strategy is to suppose that the Schauder estimates degenerate as the domains exhaust the manifold to derive a contradiction by obtaining a sequence of tensor fields that converges to a nonzero element of the (trivial) kernel of the linearized operator.  One delicate aspect in this argument is that because of the cusps of intermediate rank,  while the corresponding hyperbolic metric has constant curvature, it is not of bounded geometry since its injectivity radius vanishes.  As in \cite{Tian-Yau}, near a cusp, one can nevertheless get uniform local Schauder estimates by passing to a suitable cover, that is, quasi-coordinates in the terminology of  \cite{Tian-Yau}.  On the other hand, on say a convex co-compact hyperbolic manifold, our isomorphism theorem gives a shorter and simpler proof of the perturbation result of \cite{Lee}.  The price to pay with this method however is that one needs the  corresponding weighted H\"older spaces to lie in $L^2$.  This restricts the regularity of the boundary metric encoded as $k$ in the statement of the theorem.  

\

Even if the more detailed statement of our main result, Theorem~\ref{mt.1} below, gives some control on the behaviour of the metric at infinity, especially away from the cusps, one could ask if stronger regularity results hold.  For instance, by \cite{Anderson, CDLS2005, Helliwell,BH2014}, we know that Poincaré-Einstein metrics admit a polyhomogeneous expansion at infinity, so one could ask more generally if the geometrically finite Poincaré-Einstein metrics of Theorem~\ref{thm:main} admit a polyhomogeneous expansion with respect to the compactification of Mazzeo-Phillips.  In the hope of making progress on this question and inspired by \cite{MM1987}, we are planning in a subsequent work to develop a pseudodifferential calculus adapted to geometrically finite Poincaré-Einstein metrics.  Additionally, we expect it should be possible to relax the hypothesis of compactly supported perturbations in our result to perturbations lying in a H\"older space adapted to the foliated cusp metric at infinity.  This raises the possibility of understanding the extent to which the results of conformally compact geometry can be adapted to this more general conformally foliated cusp setting.

\

The remainder of this paper is structured as follows.  In \S \ref{0F.0} we introduce the class of $(0-\cF)$-metrics which are conformally related to geometrically finite hyperbolic metrics.  In \S \ref{s:gfhyp} we introduce geometrically finite hyperbolic metrics and their compactification to a suitable manifold with corners.  Suitable function spaces for our analysis are defined in \S \ref{s:func-spaces}, and \S \ref{s:schauder} states and proves a uniform boundary Schauder estimate for a certain exhaustion of $M$.  In \S \ref{s:linear} we prove an isomorphism theorem for Laplace operators, and study the application to our nonlinear problem.  We then prove Theorem \ref{thm:main} in \S \ref{s:proof}.

\

\textbf{Acknowledgements.}  We are happy to acknowledge useful conversations with Robin Graham, Jack Lee and Rafe Mazzeo.  EB gratefully acknowledges the hospitality of the CIRGET laboratory and the support of the Centre de Recherches Math\'ematiques through their CRM Simons Professor program, and the support of a Simons Foundation grant (\#426628, E. Bahuaud).  FR was supported by NSERC and a Canada Research chair.

\section{$(0-\cF)$-metrics} \label{0F.0}

In this section we introduce a Lie algebra of vector fields on a suitable compact manifold with corners and associated metrics that are conformal to geometrically finite hyperbolic metrics.

\

Let $X$ be a compact manifold with corners of dimension $n$.  Let $n_{ir}$ and $n_{mr}$ be non-negative integers and set $n_c = n_{ir} + n_{mr}$.  Suppose that $H_0,H_1,\ldots, H_{n_c}$ is an exhaustive list of the boundary hypersurfaces of $X$ such that 
\begin{equation}\label{0F.1}
   H_i\cap H_j=\emptyset \;  \Longleftrightarrow \; i\ne j, \quad \forall\; i,j\in\{1,\ldots,n_c\}.
\end{equation}
Moreover, the first $n_{ir}$ boundary hypersurfaces distinct from $H_0$ satisfy
\begin{equation}\label{0F.1a}
  H_0\cap H_j\ne \emptyset \quad \forall \; j \in \{1, \ldots, n_{ir}\},
\end{equation}
while the final $n_{mr}$ boundary hypersurfaces satisfy
\begin{equation}\label{0F.1b}
 H_0 \cap H_j = \emptyset \quad \forall  j \in \{n_{ir}+1, \ldots, n_c\}.
\end{equation}
Hereafter we use the convention that a greek index takes values within the index set $\{1, \ldots, n_{ir}\}$.  Thus, $X$ only has codimension $2$ corners and the boundary hypersurfaces meeting at these corners are themselves compact manifolds with boundary with 
\begin{equation}
  \pa H_0= \bigcup_{1 \leq \alpha \le n_{ir}} H_{\alpha}\cap H_0,  \quad \mbox{and} \quad \pa H_{\alpha}= H_{\alpha} \cap H_0, \; \forall \alpha \in \{1, \ldots, n_{ir}\}.
\label{0F.2}\end{equation}
Let $x_0, x_1,\ldots, x_{n_c}$ be boundary defining functions for $H_0,\ldots, H_{n_c}$ respectively.  Without loss of generality, assume that for each $i$, $x_i\equiv 1$ outside a tubular neighbourhood of $H_i$ so that for $i\ge 1$, $x_i=1$ near $H_j$ for $j\ge 1$ with $j\ne i$.  

\

For $\alpha \in \{1, \ldots, n_{ir}\}$, suppose that there is a smooth foliation $\cF_{\alpha}$ of $H_{\alpha}$ \textbf{compatible} with the boundary defining function $x_0$ in the sense that the level sets of $x_0|_{H_{\alpha}}$ near $H_0\cap H_{\alpha}$ are tangent to the foliation, that is, any leaf intersecting them is in fact contained in them.  This implies in particular that $\cF_{\alpha}$ restricts to give a foliation $\cF_{0\alpha}:= \cF_{\alpha}|_{\pa H_{\alpha}}$   on $\pa H_{\alpha}$.  Let us denote by $\cF$ the collection of foliations $(\cF_1,\ldots, \cF_{n_{ir}})$ on the boundary hypersurfaces $H_1,\ldots, H_{n_{ir}}$ of $X$.  Now, recall that on $X$, the Lie algebra of $b$-vector fields $\cV_b(X)$ consists of smooth vector fields $\xi\in \CI(X;TX)$ which are tangent to all boundary hypersurfaces of $X$.

\begin{definition}
The Lie algebra $\cV_{0-\cF}(X)$ of $(0-\cF)$-vector fields on $X$ consists of vector fields $\xi\in \cV_b(X)$ such that 
\begin{gather}
\label{0F.3a} \xi|_{H_0}=0; \\
\label{0F.3b} \xi|_{H_{\alpha}}\in \CI(H_{\alpha};T\cF_{\alpha})\; \quad \alpha \in \{1, \ldots, n_{ir}\}; \\
\label{0F.3c} \xi x_{i}\in x_{i}^2\CI(X) \; \quad  i \in \{1, \ldots, n_{c}\}.
\end{gather}
\end{definition}

Since these conditions are clearly preserved by the Lie bracket, we see that $\cV_{0-\cF}(X)$ is indeed a Lie subalgebra of $\cV_b(X)$ and $\CI(X;TX)$.  As the notation suggests, the definition of $\cV_{0-\cF}(X)$ depends on the collection of foliations $\cF$.  More subtly, it depends on the choice of the boundary defining functions $x_1,\ldots, x_{n_c}$ through the condition \eqref{0F.3c}. 

\

Near $H_0$, but away from the other boundary hypersurfaces, a $(0-\cF)$-vector field behaves like a $0$-vector field of  \cite{MM1987}.  Near $H_{\alpha}$, but away from $H_0$, it looks instead like a foliated cusp vector field (or $\cF_i$-vector field) of \cite{R2012}, the foliated version of fibered cusp vector fields of \cite{MM1998}.   Near $H_{i}$, $i \geq {n_{ir} +1}$, it looks like a cusp vector field (a fibred cusp vector field associated to a trivial fibration over a point).  At the corner $H_0\cap H_{\alpha}$, the $(0-\cF)$-vector fields are somewhat of a hybrid version of the $0$-vector fields of \cite{MM1987} and the foliated cusp vector fields of \cite{R2012}.  This can be described more precisely in local coordinates.  Given $p\in H_0\cap H_{\alpha}$, choose a coordinate chart $\varphi: \cU\to \Omega\subset \bR^n$ sending $p$ to $0$ with 
\begin{equation}
    \varphi= (x_0,x_{\alpha},y_1,\ldots, y_{b_{\alpha}},z_1,\ldots, z_{f_{\alpha}})
\label{0F.4}\end{equation}
where $x_0$ and $x_{\alpha}$ are the boundary defining functions chosen above and the coordinates $(y,z)$ are such that that the distribution $T\cF_{\alpha}$ associated to the foliation $\cF_{\alpha}$ at $x_{\alpha}=0$ is spanned by
$\frac{\pa}{\pa z_1},\ldots,\frac{\pa}{\pa z_{f_{\alpha}}}$.  In such a coordinate chart, $(0-\cF)$-vector fields are of  the form
\begin{equation}
   a x_0x_{\alpha}\frac{\pa}{\pa x_0} + bx_0x_{\alpha}^2\frac{\pa}{\pa x_{\alpha}}+ \sum_{j=1}^{b_{\alpha}} c_jx_0x_{\alpha} \frac{\pa}{\pa y_j} + \sum_{k=1}^{f_{\alpha}} d_k x_0\frac{\pa}{\pa z_k}
\label{0F.5}\end{equation}
for $a,b, c_1,\ldots,c_{b_{\alpha}}, d_1,\ldots, d_{f_{\alpha}}$  arbitrary smooth functions.  Indeed, the factor $x_0$ in each term ensures that condition \eqref{0F.3a} is satisfied, while the factors of $x_{\alpha}$ and $x_{\alpha}^2$ ensure that conditions \eqref{0F.3b} and \eqref{0F.3c} are satisfied.  

\

This shows that $\cV_{0-\cF}(X)$ is a locally free $\CI(X)$-module of rank $n$.  Let  ${}^{0-\cF}TX\to X$ be the vector bundle with fiber above $p\in X$ given by 
$$
            {}^{0-\cF}T_pX:= \cV_{0-\cF}(X)/ \cI_p\cV_{0-\cF}(X) 
$$  
where $\cI_p\subset \CI(X)$ is the ideal of smooth functions vanishing at $p$.  By the Serre-Swan theorem, this vector bundle comes with a canonical morphism of vector bundles 
\begin{equation}
   \mathfrak{a}: {}^{0-\cF}TX\to TX
\label{0F.6}\end{equation} 
inducing the identification
\begin{equation}
     \CI(X;{}^{0-\cF}TX)=\cV_{0-\cF}(X).
\label{0F.7}\end{equation}
The map \eqref{0F.6} induces an isomorphism ${}^{0-\cF}TX|_{X\setminus \pa X}\cong TX|_{X\setminus \pa X}$, but fails to be an isomorphism on $\pa X$.  We remark that the vector bundle ${}^{0-\cF}TX\to X$ is naturally a Lie algebroid with anchor map given by \eqref{0F.6} and Lie bracket on its sections induced by the identification \eqref{0F.7}, although this fact will not be used in this paper.

\

In the terminology of \cite{ALN2004}, $\cV_{0-\cF}(X)$ is a Lie structure at infinity.  There is therefore a natural class of metrics associated to it.  We say a \textbf{$(0-\cF)$-metric} on $X\setminus \pa X$ is a Riemannian metric corresponding to the restriction of a bundle metric on ${}^{0-\cF}TX$ to $T(X\setminus \pa X)$ via the identification
$T(X\setminus \pa X)\cong {}^{0-\cF}TX|_{X\setminus \pa X}$ induced by the anchor map \eqref{0F.6}.  \label{0F.9}

\

In the coordinate chart \eqref{0F.4}, an example of $(0-\cF)$-metric is given by
$$
       \frac{dx_0^2}{(x_0x_{\alpha})^2}+ \frac{dx_{\alpha}^2}{(x_0x_{\alpha}^2)^2}+ \sum_{j=1}^{b_{\alpha}} \frac{dy_j^2}{(x_0x_{\alpha})^2}+ \sum_{k=1}^{f_{\alpha}}\frac{dz_k^2}{x_0^2}.
$$

Since $(0-\cF)$-metrics come from a Lie structure at infinity, we know from \cite{ALN2004} that they are complete of infinite volume and that they are all quasi-isometric to each other.  Furthermore, it is shown in \cite{ALN2004} that their curvature and all its derivatives are bounded.  On the other hand, one can check directly in local coordinates that the injectivity radius is positive, so that $(0-\cF)$-metrics are of bounded geometry.  

\

The class of metrics we are interested is not quite the one of $(0-\cF)$-metrics, but a conformal cousin.  Define a \textbf{$(0-\cF)_c$-metric} $g_{(0-\cF)_c}$ on $X\setminus \pa X$ to be a metric of the form
$$
     g_{(0-\cF)_c}:= \left(\prod_{i=1}^{n_c}x_i\right)^2 g_{0-\cF}
$$
for some $(0-\cF)$-metric $g_{0-\cF}$. \label{0F.11} In the coordinate chart \eqref{0F.4}, an example of $(0-\cF)_c$-metric is given by
\begin{equation}
       \frac{dx_0^2}{x_0^2}+ \frac{dx_{\alpha}^2}{(x_0x_{\alpha})^2}+ \sum_{j=1}^{b_{\alpha}} \frac{dy_j^2}{(x_0)^2}+ \sum_{k=1}^{f_{\alpha}}\frac{x_{\alpha}^2dz_k^2}{x_0^2}.
\label{0F.12b}\end{equation}\label{0F.12}

Despite this extra conformal factor, a $(0-\cF)_c$-metric is still complete of infinite volume, but it becomes of finite volume outside a tubular neighbourhood of $H_0$.  More importantly, the injectivity radius is no longer positive and the curvature is typically not bounded.  Near $H_0$ but away from $H_i$ for $i\ge 1$, it still looks like a $0$-metric.  Near $H_{\alpha}$ for $\alpha \in \{1, \ldots, n_{ir}\}$ but away from $H_0$, it looks instead like a foliated cusp metric of \cite{R2012}, and near $H_i$ for $i \geq n_{ir}+1$ it looks like a cusp metric.

\

In the next section we explain how geometrically finite hyperbolic metrics provide examples of $(0-\cF)_c$-metrics.

\section{Geometrically finite hyperbolic metrics}
\label{s:gfhyp}

In this section we define the class of hyperbolic metrics of interest and fix notation.  Our description of a suitable compactification of these hyperbolic metrics to a manifold with corners was first stated by Mazzeo and Phillips \cite{MazzeoPhillips} and we refer the reader to a very thorough discussion there.  Compare also \cite{GuillarmouMazzeo} and \cite{GuillarmouMoroianuRochon}.

\

We consider the simply connected hyperbolic space of constant curvature $-1$, $\bH^n$, represented by the Poincar\'e ball model.  In what follows, $\Gamma$ will be a discrete, torsion-free group of isometries acting properly discontinuously on $\bH^{n}$.  Recall that (nontrivial) isometries of $\bH^n$ are classified as elliptic, parabolic or hyperbolic if they fix $0$, $1$ or $2$ points on the boundary sphere at infinity, $\bS^{n-1}_{\infty}$.  Our interest will be when $\Gamma$ is geometrically finite with cusps, which as we explain will entail that $\Gamma$ contains parabolic elements and that an appropriate quotient of a subset of $\overline{\bH^n}$ may be written as a union of a compact region and finitely many disjoint cusp neighbourhood ends.  When $n > 3$, this is more general that requiring that $\Gamma$ possesses a fundamental domain for the action bounded by totally geodesic hyperplanes and portions of the sphere at infinity, see \cite{Bowditch}.

\

Given any $p \in \bH^n$, the set of accumulation points $\Lambda_{\Gamma}$ (in $\overline{\bH^n} = \bB^n \cup \bS^{n-1}_{\infty}$) of the orbit $\Gamma \cdot p$ is independent of the choice of $p$ and is called the limit set of $\Gamma$.  The action of $\Gamma$ on $\bH^n$ extends to a properly discontinuous action on $\Omega_{\Gamma} = \bS^{n-1}_{\infty} \setminus \Lambda_{\Gamma}$, and $\Gamma$ acts properly discontinuously when we adjoin this conformal boundary $\overline{M} := (\bH^{n} \cup \Omega_{\Gamma}) / \Gamma$.  Note that $\Mbar$ is now a manifold with noncompact boundary.  It is this latter manifold we now decompose further.

\

Now fix a parabolic element $\gamma \in \Gamma$.  Let $p \in \bS^{n-1}_{\infty}$ be the fixed point of $\gamma$.  This point gives rise to a cusp.  Let $\Gamma_p$ be the subgroup of elements of $\Gamma$ that fix $p$.   Consider the upper-half space model of hyperbolic space $\bH^n=\{(x,y^1,\ldots,y^{n-1}: x > 0 \}$ with the fixed point $p$ placed at infinity $x=+\infty$.  Each level set of $x$, $E_{a} := \{ x = a \}$ is a horosphere on which the elements of $\Gamma_p$ acts by euclidean isometries, and as such there is a maximal normal free Abelian subgroup of $\Gamma_p$ with finite index $f$; the number $f$ is called the rank of the cusp.  We distinguish two cases: when the cusp has maximal rank $n-1$ and when the cusp has intermediate rank $1 \leq f < n-1$.  

\

As explained in \cite{GuillarmouMazzeo, MazzeoPhillips}, it is possible to find a maximal subspace $\bR^{f} \subset \bR^{n-1} = E_1$, $1 \leq f \leq n-1$ that is mapped to itself under $\Gamma_{p}$.  For a certain polyhedral fundamental domain $K \subset \bR^f$, the fundamental domain for the action of $\Gamma_p$ on $\bR^{n-1}$ becomes $\bR^{n-1-f} \times K$, where the action on the $\bR^{n-1-f}$ factor is through rotation.  The quotient $\bR^{n-1} / \Gamma_p$ is then the total space of a flat vector bundle over a compact manifold $K/\Gamma_p$.  Furthermore, $\bR^+  \times \bR^{n-1-f} \times K$ is a fundamental domain for the action of $\Gamma_p$ on $\bH^n$.  Returning to the splitting $\bH^n = \bR^+  \times \bR^{n-1-f} \times \bR^f$, define
\[ C_{p}(R) := \{ (x,y,z) \in [0,\infty) \times \bR^{n-1-f} \times \bR^f: x^2 + |y|^2 \geq R\} \subset \overline{\bH^n}. \]
The set $C_p(R)$ is invariant under $\Gamma_p$ and convex with respect to the hyperbolic metric.  In a similar way we define $C_p(R)$ for all other parabolic fixed points.  For each parabolic fixed point $p$, there exists $R > 0$ sufficiently large so that
\begin{enumerate}
\item $C_p(R) \subset \bH^n \cup \Omega_{\Gamma}$ and
\item For any $\gamma \in \Gamma \setminus \Gamma_p$, $\gamma C_p(R) \cap C_p(R) = \emptyset$.
\end{enumerate}

As a consequence, $C_p(R)$ descends to a set $\mathcal{C}_p = (\cup_{\gamma} \gamma C_p(R)) / \Gamma$ which has interior isometric to $C_p(R) / \Gamma_p$.  We call this the standard cusp region $\mathcal{C}_p$ associated to the orbit of $p$; \cf \cite[Section 3.1]{Bowditch}.  Thus we say $M := \bH^n / \Gamma$ is geometrically finite if $\Mbar = (\bH^n \cup \Omega_{\Gamma}) / \Gamma$ has a decomposition into the union of a compact set $\mathcal{C}_0$ and a finite number $n_c$ of standard cusp regions, with $n_c \geq 1$:
\[ \Mbar = (\bH^n \cup \Omega_{\Gamma}) / \Gamma = \mathcal{C}_0 \cup \left( \bigcup_{j=1}^{n_c} \mathcal{C}_{p_j} \right). \]
Note that the compact set $\mathcal{C}_0$ contains points on $\partial \Mbar=\Omega_{\Gamma}/\Gamma$.  Moving forward we assume that the enumeration of the cusps is now fixed, and with $n_{ir}$ intermediate rank cusps followed by  $n_{mr}$ cusps of maximal rank so that $n_c = n_{ir} + n_{mr}$.
We now describe the compactification by Mazzeo-Phillips \cite{MazzeoPhillips} of $\Mbar$ to a manifold with corners.

\

We start with the compactification of a single intermediate rank cusp region, $\mathcal{C}_p$ by working with a representative $C_p$.  Introduce coordinates $(x,y^1,\ldots,y^{n-1-f},z^1,\ldots,z^f) \in \bR^+ \times \bR^{n-1-f} \times K$, and recall that the cusp is at $x=+\infty$.  In each fibre $z = z_0$, invert in the unit sphere through the map $(u,v^i) = (\frac{x}{x^2 + |y|^2}, \frac{y^i}{x^2 + |y|^2})$.  This places the cusp at $(u,v) = (0,0)$ and the cusp region of the manifold is now diffeomorphic to $\bB^{n-f}_{+} \times K$.  In this region the hyperbolic metric is expressed as 
\begin{equation} \label{eqn:hyperbolic-uvcoords}
 h = \frac{du^2 + \sum_{i=1}^{n-1-f} (dv^i)^2 + (u^2+|v|^2)^2 \sum_{i=1}^f(dz^i)^2}{u^2}.
\end{equation}

We blow up $\{0\}\times \{0\}\times \bR^f$ in $\bR^+\times \bR^{n-1-f}\times \bR^f$  by introducing polar coordinates $r = \sqrt{u^2 + |v|^2}$, $\phi \in \bS^{n-1-f}_+$, with $\phi^0 = u/r > 0, \phi^i = v^i/r$.  It will be convenient to write the coordinate $\phi^0$ as $\rho$ instead.  Thus the blown up cusp neighbourhood is diffeomorphic to $\bR^+_r \times \bS^{n-1-f}_{+} \times \bR^f$ and we may compactify to a manifold with corners by adjoining the boundary faces $r = 0$ (the cusp face) and $\rho = 0$ (the $0$-face), obtaining $[0,\infty)\times \overline{\bS}^{n-1-f}_+\times \bR^f$, where $\overline{\bS}^{n-1-f}_+$ is the closure of $\bS_+^{n-1-f}$ in $\bR^{n-f}$.  Notice that $\Gamma_p$ acts naturally on this space.
Passing to the quotient, we have compactified $\mathcal{C}_p$ by
\begin{equation}\overline{\mathcal{C}}_p := [0,\vep)_r \times ((\overline{\bS}^{n-1-f}_+)_{\phi} \times \bR^f)/\Gamma_p,
\label{mc.1a}\end{equation}
a hemisphere bundle over  $[0,\vep)_r\times \bR^f/ \Gamma_p$.  
 We use $\{w^1, \ldots, w^f\}$ to represent local coordinates on $\bR^f/ \Gamma_p$.  In a local trivialization of the hemisphere bundle on $\bR^f/\Gamma_p$,  hyperbolic metric can then be written
\begin{equation} \label{mc.1}
h|_{\mathcal{C}_p} = \frac{dr^2}{r^2 \rho^2} + \frac{d\phi^2}{\rho^2} + \frac{r^2}{\rho^2} dw^2
\end{equation}

where $dw^2$ is a flat metric and $d \phi^2$ is the round metric on the $(n-1-f)$-sphere.  Observe that the function $u = r\rho$ descends to $[0,\vep)_r \times ((\bS^{n-1-f}_+)_{\phi} \times \bR^f)/\Gamma_p$ and provides a total boundary defining function at this end.

\

If the cusp is of maximal rank, then the fundamental domain for $\Gamma_p$ intersects the level set $E_1$ in a compact polyhedron $K \subset \bR^{n-1}$ that is a fundamental domain for the action of $\Gamma_p$ on $\bR^{n-1}$.  We obtain a compact flat manifold $N = \bR^{n-1} / \Gamma_p$, and then $\mathcal{C}_p$ is diffeomorphic to $[R,\infty)\times N$.  Set $r = 1/x$ in this neighbourhood and then adjoin the face $r = 0$ to obtain the compactification
$$
  \overline{\mathcal{C}}_p=\left[0,\frac{1}R\right]_r\times N.
$$ 
In the region $\mathcal{C}_p$, the hyperbolic metric may be written
\begin{equation} \label{mc.2}
h|_{\mathcal{C}_p} = \frac{dr^2}{r^2} + r^2 g_N
\end{equation}
with $g_N$ a flat metric on $N$.

\

Finally, near any point in $\mathcal{C}_0$, the compactification is already provided by $\overline{M}=(\bH^n\cup \Omega_{\Gamma})/\Gamma$ with boundary face 
given by $\Omega_{\Gamma}/\Gamma$.  Away from the cusps, say for $U$ an open set in $\Omega_{\Gamma}/\Gamma$ with closure not intersecting the cusps, the metric near this boundary can be put in the form
\begin{equation}
         h= \frac{d\rho^2+ h_{U}(\rho)}{\rho^2}    \quad \mbox{in} \quad  [0,\delta)_{\rho} \times U
\label{mc.3}\end{equation}
with $h_{U}(\rho)$ a smooth family of metrics on $U$ parametrized by $\rho\in [0,\delta)$ with $\rho=0$ corresponding to the conformal boundary at infinity.  Note that $\rho$ is a special boundary defining function here since $|d\rho|^2_{\rho^2h}\equiv 1$ near $\rho=0$.   Notice however that in the cusp neighbourhood $\mathcal{C}_p$, the manifold $\Omega_{\Gamma}/\Gamma$ is itself compactified by adding a boundary, namely 
$(\pa \overline{\bS}^{n-1-f}_+\times \bR^{f})/\Gamma_p$.

\

Collecting these observations, we have thus obtained a compactification of $\Mbar = (\bH^n \cup \Omega_{\Gamma}) / \Gamma$ to a smooth manifold with corners $\Mdbar$ with boundary hypersurfaces $H_0, H_1, \ldots, H_{n_c}$, where 
$H_0 = \overline{\Omega_{\Gamma}/\Gamma}$ is the closure of $\Omega_{\Gamma}/\Gamma$ in $\Mdbar$, and $H_i=\{0\}\times (\overline{\bS}^{n-f-1}\times \bR^f)/\Gamma_p$ is the $i$-th cusp face defined by $\{ r_i = 0\}$ in $\mathcal{C}_{p_i}$ for some $p_i\in \bS^{n-1}_{\infty}$.  Note that for each intermediate rank cusp face $H_{\alpha}$, $H_0 \cap H_{\alpha}$ is codimension two corner, and that $H_i \cap H_j = \emptyset$ for $i \neq j$.  Each compactified intermediate rank cusp neighbourhood $\overline{\mathcal{C}}_{p_\alpha}$ is the total space of a flat bundle,
\begin{equation}
 \Phi_i: \overline{\mathcal{C}}_{p_\alpha} \longrightarrow D_{\alpha}
 \label{fb.1}\end{equation}
over a compact flat manifold $D_{\alpha}$ that yields a cusp of rank $f_{\alpha} = \dim D_{\alpha}$.  Since the bundle \eqref{fb.1} is flat, its horizontal distribution is integrable, so induces a foliation on $H_{\alpha}$ that we will denote $\cF_{\alpha}$.      In terms of the decomposition \eqref{mc.1a}, a leaf of $\cF_{\alpha}$ is the image under the quotient map by the action of $\Gamma_{p_{\alpha}}$ of a submanifold of the form
$$
             \{0\}\times \{q\} \times \bR^{f_{\alpha}}\subset \{0\}\times (\overline{\bS}^{n-1-f_{\alpha}}_+)_{\phi} \times \bR^{f_{\alpha}}
$$
for some $q\in (\overline{\bS}^{n-1-f_{\alpha}}_+)_{\phi}$.

\

In what follows we may assume that the $r_i$'s are globally defined by smoothly truncating its value to be $1$ outside of its cusp neighbourhood.  We continue to let $\rho$ be a defining function for $H_0$ which corresponds to $\rho$  in the coordinate chart of \eqref{mc.1} for each $\alpha$.  The product 
\begin{equation} \label{eqn:tbdf}
\sigma = \rho \prod_{i=1}^{n_c} r_i 
\end{equation} 
furnishes a smooth global boundary defining function, and for sufficiently small $\vep > 0$ the superlevel sets
\begin{equation} \label{eqn:exhaustion} M_{\vep} = \{ p: \sigma(p) \geq \vep \} \end{equation}
furnish a compact exhaustion of $\Mdbar\setminus\pa \Mdbar$ by smooth compact manifolds with boundary.

\

We now connect these hyperbolic metrics to the metrics defined in the previous section.
\begin{lemma}
The hyperbolic metric on $\Mdbar\setminus \pa \Mdbar$ is a $(0-\cF)_c$-metric with respect to the $(0-\cF)$-Lie structure at infinity on 
$\Mdbar$ induced by boundary defining functions $\rho,r_1,\ldots, r_{n_c}$ and the collection of foliations $\cF=(\cF_1,\ldots, \cF_{n_{ir}})$.
\label{fb.2}\end{lemma}
\begin{proof}
It suffices to notice that, near each $H_{\alpha}$ for $\alpha \in \{1, n_1, \ldots, n_{ir}\}$, the hyperbolic metric is precisely of the form \eqref{0F.12b} by \eqref{mc.1}, with similar observations for the other hypersurfaces.
\end{proof}

Finally, we remark that as a $(0-\cF)_c$-metric, the hyperbolic metric is very special, since it has constant sectional curvature, which implies that its curvature and all derivatives are bounded.  However it is not of bounded geometry since its injectivity radius is zero.  Thus when obtaining estimates in weighted H\"older spaces we will pass to a cover of positive injectivity radius and consequently of bounded geometry, see \S \ref{s:schauder}.       

\

In the next section we extend the notion of conformally compact asymptotically hyperbolic metric to a broader class modeled on these metrics.

\section{Function spaces} \label{s:func-spaces}

Let $(M,h)$ be a geometrically finite hyperbolic metric with cusps as described in the previous section.  Recall that an enumeration of the cusps is fixed.  Through the paper we write $\Sigma^2(M)$ for the bundle of symmetric $2$-tensors and $\Sigma^2_0(M)$ for the elements of $\Sigma^2(M)$ that are also trace-free with respect to the metric.  

\

For the smooth tensor bundle of $(p,q)$-tensors, $E := (TM)^{\otimes p} \otimes (T^*M)^{\otimes q} \longrightarrow M$, we have the usual space of $L^2$ sections, defined as the completion of compactly supported smooth sections  with respect to the $L^2$ norm
\[ \|u\|_{L^2} := \left( \int_M |u|_h^2 dv_h \right)^{1/2}.\]

\

To describe weighted versions of these spaces, encode weighting with respect to a vector \linebreak $\mu=(\mu_0,\mu_1,\ldots, \mu_{n_c}) \in \bR^{1+n_c}$ by 
\[ \sigma^{\mu} := \rho^{\mu_0} r_1^{\mu_1} \ldots r_{n_c}^{\mu_{n_c}}. \]
Note that we write inequalities like $\mu \geq 0$ as shorthand for $\mu_i \geq 0$ for all $i = 0, \ldots, n_c$.  Define
$$L^2_{\mu}(M;E) = \sigma^{\mu} L^2(M;E)
$$ 
with norm $\| u \|_{L^2_{\mu}} := \| \sigma^{-\mu} u \|_{L^2}.$

\

The following criterion will be useful to determine  when a section of $E$ is square integrable.

\begin{lemma} \label{l2-cut}
Suppose a continuous section $u: M \to E$ satisfies $|u|_h \leq C \sigma^{\mu}$ for some $C > 0$.  If 
\[ \mu_0 > \frac{n-1}{2}, \; \mu_j > -\frac{f_j}{2}, \; \forall j \in \{1, \ldots, n_c\} , \]
where $f_j$ is the rank of the $j$-th cusp, then $u \in L^2(M;E)$.
\end{lemma}
\begin{proof}
It suffices to check that  the $L^2$ norm is finite both in the neighbourhood $\sC_0$, where $\sigma = \rho$ and in any cusp neighbourhood $\overline{\sC}_{p_j}$ where $\sigma = \rho r_j$.  Here we outline the calculation for the cusp neighbourhood as the calculation in the regular neighbourhood is simpler.  In view of the expression for the metric given in equation \eqref{mc.1}, one may compute an expression for the volume form and then
\begin{align*}
\int_{\sC_{p_j}} |u|_h^2 dv_h &\leq C \int_0^* \int_0^* \rho^{2 \mu_0} r_j^{2 \mu_j} \frac{r_j^{f_j-1}}{\rho^n}  dr_j d\rho
\leq C \int_0^*  \rho^{2 \mu_0 - n} d\rho \int_0^* r_j^{2 \mu_j+f_j-1} dr_j.
\end{align*}
This integral is finite if $\mu_0 > (n-1)/2$ and $\mu_j > -f_j/2$.
\end{proof}

\

For any non-negative integer $k$, let $C^k(M,E)$ be the space of  $k$-times continuously differentiable sections $u$ of $E$ such that the norm
\[ \|u\|_{C^k(M,E)} := \sum_{\ell=0}^k \sup_{M} | (\nabla^h)^{\ell} u |_{h} \]
is finite. 

\

Before introducing the H\"older quotient, by a path we will mean any piecewise $C^1$ map $\gamma: [0,1] \to M$.  The length of a path is given by
\[ \mathrm{len(\gamma)} = \int_0^1 | \gamma'(s) |_h \; ds. \]
In what follows we let $P_{\gamma}: E|_{\gamma(0)} \longrightarrow E|_{\gamma(1)}$ denote  parallel transport of sections of $E$ along $\gamma$ with respect to the metric $h$.

\

For any non-negative integer $k$ and $\alpha \in (0,1]$ we define the $(k,\alpha)$-norm on $C^k$ sections of $E$ by
\[ \| u \|_{k,\alpha} = \| u \|_{C^{k,\alpha}(M;E)} := [u]_{k,\alpha} + \| u \|_{C^{k}(M;E)} , \]
where the H\"older seminorm is defined by
\[ [u]_{k,\alpha} := \sup \left\{ \frac{ \left|P_{\gamma} \left(\nabla^k u( \gamma(0) ) \right) - \nabla^k u( \gamma(1) ) \right|_h }{\mathrm{len}( \gamma )^{\alpha}}: \gamma \; \mbox{is a path with } \gamma(0) \neq \gamma(1) \right\}. \]
The space $C^{k,\alpha}(M;E)$ is then defined as the subspace of $C^k$ sections of $E$ for which the norm $\|\cdot\|_{k,\alpha}$ is finite.  

\

Note that 
\[ \nabla^h: C^{k,\alpha}(M;E)  \xrightarrow{\phantom{quebec}} C^{k-1,\alpha}(M;T^*M \otimes E) \]
is a bounded operator.  In view of the metric decomposition of equation \eqref{mc.1} in each cusp neighbourhood described in the previous section, requiring that a section lie in a H\"older space imposes some restriction on the asymptotic behaviour of that section at each boundary hypersurface.  For example, for a function $u: M \longrightarrow \bR$ restricted to the $j$-th cusp neighbourhood, boundedness of $\sup |\nabla^h u|_h$ is equivalent to separate boundedness of the $(0-\cF)_c$-derivatives
\begin{equation} \label{eqn:0cusp}
 r\rho \frac{\partial u}{\partial r}, \rho \frac{\partial u}{\partial \phi}, \frac{\rho}{r} \frac{\partial u}{\partial w},
\end{equation}
where we abbreviated $r = r_j$.  If $\frac{\rho}{r} \frac{\partial u}{\partial w}$ is bounded, then $\rho \partial_w u$ must vanish as $r \to 0$ and $\rho > 0$, and thus $\partial_w u$ extends to be $0$ along $r=0$ when $\rho > 0$.  Thus, if $u$ has a $C^1$ extension to $\Mdbar$, then the restriction of $u$ to a leaf of $\cF_j$ gives a constant function.

\

We additionally define a weighted H\"older space $C^{k,\alpha}_{\mu}(M;E) := \sigma^{\mu} C^{k,\alpha}(M;E)$ with norm
\[ \| u\|_{k,\alpha,\mu} := \| \sigma^{-\mu} u \|_{k,\alpha}. \] 

\

The following lemma gives some of the basic properties of these H\"older spaces.

\begin{lemma}
Let $k \in \bN_0$, $\alpha \in [0,1]$.  Let $E_1, E_2$ denote tensor bundles over $M$.
\begin{enumerate}
\item The hyperbolic metric $h$, its inverse $h^{-1}$ and any covariant derivative of curvature $\nabla^{\ell} \mathrm{R}^h$ lie in $\displaystyle \bigcap_{m \in \bN} C^{m,\alpha}_0$.
\item Contracting with $h$ or $h^{-1}$ preserves weight and regularity.
\item Pointwise tensor product induces a continuous map
\[ C^{k,\alpha}_{\mu}(M;E_1) \times C^{k,\alpha}_{\mu'}(M,E_2) \xrightarrow{\phantom{quebec}} C^{k,\alpha}_{\mu + \mu'}(M; E_1 \otimes E_2). \]
\item The covariant derivative induces a continuous map
\[ \nabla^h: C^{k,\alpha}_{\mu}(M;E_1)  \xrightarrow{\phantom{quebec}} C^{k-1,\alpha}_{\mu}(M;T^*M \otimes E_1). \]
\item For any non-negative weight $\mu$ and any $\ell\in\bN_0$, $\sup_M | (\nabla^h)^{\ell} \sigma^{\mu} |$ is bounded.  In particular, for two weights $\mu_1$ and $\mu_2$ with $\mu_1\ge \mu_2$, there is a continuous inclusion
$$
         C^{k,\alpha}_{\mu_1}(M;E_1) \subset C^{k,\alpha}_{\mu_2}(M;E_1).
$$
\end{enumerate}
\end{lemma}
\begin{proof}
Items (a)--(d) are straightforward.  

\

Item (e) stems from our particular choice of total boundary defining function given in the previous section.  Recall that $\sigma$ and the choice of coordinates in each cusp neighbourhood is adapted to the geometry.  In particular in the $j$-th cusp neighbourhood, $\sigma = \rho r_j$ is constant along the leaves of $\cF_j$, thus the action of the singular vector field $\frac{\rho}{r} \frac{\partial}{\partial w}$ on $\sigma$ always vanishes, which need not be true for an arbitrary total defining function.  Thus for any $\ell\in\bN_0$ and for any non-negative weight, $\sup_M | (\nabla^h)^{\ell} \sigma^{\mu} |$ is bounded.  For weights $\mu_1$ and $\mu_2$ such that $\mu_1\ge \mu_2$, this implies that there is a continuous inclusion 
$$
         C^{k,\alpha}_{\mu_1}(M;E_1) \subset C^{k,\alpha}_{\mu_2}(M;E_1).
$$
\end{proof}

With these preliminaries in hand, we now extend the notion of conformally compact asymptotically hyperbolic metric.  For motivation we briefly review two well-known types of metrics.  

\

First, if $(\Mbar,\gbar)$ is a smooth compact manifold with boundary, and $\rho$ is a smooth boundary defining function, then the conformally compact metric $g = \rho^{-2} \gbar$ is always asymptotically negatively curved near $\partial M$. If the conformally invariant condition that $|d\rho|^2_{\rho^2 g} = 1$ on $\partial M$ is satisfied, then all sectional curvatures of $g$ tend to $-1$ plus corrections of order $\rho$.  Moreover $(M,g)$ is complete and of bounded geometry.  The facts above are due to Mazzeo \cite{Mazzeo88}.

\

Now given a hyperbolic metric $h$  as in equation \eqref{mc.3} with accompanying open subset $U \subset \Omega_{\Gamma} / \Gamma$ away from the cusps, the discussion above shows that we may take perturbations by symmetric $2$-tensors ``of the same order'' as the metric with respect to the defining function, subject only to $|d\rho|^2_{\rho^2 g} = 1$ at $\rho = 0$.  In particular, any ``tangential'' perturbation of the form
\[  g = \frac{d\rho^2+ h_{U}(\rho)}{\rho^2}+ \frac{e_{U}(\rho)}{\rho^2} \]
will be asymptotically hyperbolic, where $e_U(\rho)$ is a smooth family of small 2-tensors on $U$ parametrized by $\rho$.  Note that $|e_U|_g = O(\rho^2)$.

\

Second, consider a hyperbolic cusp metric in a compactified neighbourhood, as in equation \eqref{mc.1}.  This metric has bounded curvature only when the induced metric on the leaves of $\cF_i$ is flat.  Thus in these neighbourhoods, we must always take perturbations which vanish to some positive order.

\

Now suppose that $(M,h) := (\bH^n / \Gamma,h)$ is a geometrically finite hyperbolic metric with $n_{ir}$ intermediate rank cusps and $n_{mr}$ maximal rank cusps, and consider the compactification $\Mdbar$ described in the previous section.   We say that a Riemannian metric $g$ on $\Mdbar\setminus \pa \Mdbar$ is asymptotically hyperbolic with cusps modeled on $(M,h)$ if there is a weight vector $\mu \in \bR^{1+n_c}$ with $\mu_0 \geq 0$ and $\mu_i > 0$ for $i = 1, \ldots, n_c$ so that $g$ is a $(0-\cF)_c$-metric of the form
\[ g = h + e, \]
where $e$ symmetric $(0-\cF)_c$ $2$-tensor with $e \in C^{k,\alpha}_{\mu}(M;\Sigma^2(M))$, and $|d\rho|^2_{\rho^2 g} = 1$ on $H_0$.  When necessary we will write that $g$ is $(k,\alpha,\mu)$-asymptotically hyperbolic to emphasize the parameters involved, always keeping in mind the fixed geometrically finite hyperbolic reference metric.

\

As the leading part of $g$ is $h$, it will be useful to perform analysis of $g$ in the function spaces defined with respect to $h$.  Note that $g$ is quasi-isometric to $h$.  We additionally have

\begin{prop} Let $g$ be a $(k,\alpha,\mu)$-asymptotically hyperbolic metric with cusps as above.  Then for $\ell\leq k$,
\[ \nabla^g: C^{\ell,\alpha}_{\mu}(M;E) \xrightarrow{\phantom{quebec}} C^{\ell-1,\alpha}_{\mu}(M;T^*M \otimes E) \]
is bounded.
\end{prop}
\begin{proof}
First, in any system of local coordinates, write the Christoffel symbols of $h$ and $g$ as $\hat{\Gamma}_{ij}^p$ and $\Gamma_{ij}^p$ respectively.  Recall that 
\[ A_{ij}^p = \hat{\Gamma}_{ij}^p - \Gamma_{ij}^p \]
define the components of a well-defined $(1,2)$-tensor field.  Evaluating this tensor in $h$-normal coordinates at a point allows us to see that
\[ A_{ij}^p = -\frac{1}{2} g^{pm} ( \nabla^h_i e_{jm} + \nabla^h_j e_{im} - \nabla^h_m e_{ij}). \]
In particular if $g$ is $(k,\alpha,\mu)$-asymptotically hyperbolic then $\nabla^h e$ is  $C^{k-1,\alpha}_{\mu}$, and thus $A$ is  $C^{k-1,\alpha}_{\mu}$.

\

If $u$ is a $(p,q)$-tensor, then the formula for the covariant derivative with respect to $g$ may be expressed in terms of the covariant derivative of $h$ and various contractions with $A$.  We write this abstractly as
\[ \nabla^g u = \nabla^h u + A * u, \]
where $A*u$ denotes linear combinations of contractions of $A$ with $u$. Thus if $u \in C^{\ell,\alpha}_{\mu}(M;E)$, each term on the right hand side of this equation lies in $C^{\ell-1,\alpha}_{\mu}(M;E)$, since $\min\{\ell-1,k-1\} = \ell-1$ when $\ell \leq k$.
\end{proof}

As mentioned above, a generic cusp metric has unbounded curvature when the link is not flat.  As we now prove, since our asymptotically hyperbolic metrics are hyperbolic at leading order and we always take the cusp weight such that $\mu_i > 0$, $i>0$, we obtain metrics with bounded curvature.

\begin{prop} \label{prop:Rm-bdd} Let $k \geq 2$ and define $\mathcal{U} := \{ e \in  C^{k,\alpha}_{\mu}(M;\Sigma^2(M)): h+e \; \mbox{is positive definite}\}$.  Then for $e \in \mathcal{U}$, 
 $g = h + e$ has bounded curvature and the nonlinear operator
\begin{align*}
\mathcal{U} & \xrightarrow{\phantom{quebec}} C^{k-2,\alpha}_{\mu}(M; \Sigma^2(M)) \\
 e & \xmapsto{\phantom{quebec} }\Rc(h+e) + (n-1)(h+e)
 \end{align*}
is a smooth map of Banach manifolds.
\end{prop}  
\begin{proof}
Consider the Riemann curvature $(1,3)$-tensor of $g$, $\riem^g$, which we may write  abstractly
\[ \riem^g = \partial \Gamma + \Gamma * \Gamma, \]
where $\Gamma$ represents the Christoffel symbols with respect to $g$.  Interpolating the Christoffel symbols $\hat{\Gamma}$ of the reference hyperbolic metric $h$ and using $A$ as the difference tensor as in the proof of the previous proposition, one obtains abstractly
\begin{align} \label{eq:riem-curv-g}
 \riem^g &= \riem^h+ \partial A + A*\hat{\Gamma} + A*A \nonumber \\
 &= \riem^h+\nabla^h* A + A*A,
\end{align} 
where in the second equality we express the tensor in $h$-normal coordinates at a point, in order to recognize a $h$-covariant derivative, and the asterisk ($*$) indicates linear combinations of the quantities indicated, whose explicit form is unimportant for what follows.

\

As in the previous proof, $A \in C^{k-1,\alpha}_{\mu}$, and thus $\nabla^h * A \in 
C^{k-2,\alpha}_{\mu}$, and $A*A \in C^{k-1,\alpha}_{2\mu}$.  This shows that $g$ has bounded curvature.

\

Contracting equation \eqref{eq:riem-curv-g} and adding a factor of $(n-1) g$ to both sides we obtain
\begin{align*} 
 \Rc(g) + (n-1)g &= \Rc(h) + (n-1) (h+e) + \nabla^h * A + A*A \\
 &= (n-1) e+\nabla^h * A + A*A,
\end{align*}
which lies in $C^{k-2,\alpha}_{\mu}$ once more.  The smoothness of $\Rc$ follows from the fact that the Ricci tensor is a polynomial contraction of the inverse metric and its first two derivatives.
\end{proof}

Using this proposition and similar estimations, one may check a similar mapping property for the gauge-adjusted Einstein equation defined in equation \eqref{definition-Q}.
\begin{cor} \label{cor:Qmapping} 
The nonlinear operator
\begin{align*}
\mathcal{U} \times \mathcal{U}  & \xrightarrow{\phantom{quebec}} C^{k-2,\alpha}_{\mu}(M; \Sigma^2(M)) \\
 (e_1,e_2) & \xmapsto{\phantom{quebec} }Q(h+e_1,h+e_2)
 \end{align*}
is a smooth map of Banach manifolds.
\end{cor}

\section{Boundary Schauder estimates for the exhaustion}
\label{s:schauder}

The proof of the isomorphism theorem for operators of the form $\Delta^h + K$ (with $K$ a constant) requires boundary Schauder estimates applied over an exhaustion of $M$ by the superlevel sets
\[ M_{\vep} := \{ p \in M: \sigma(p) \geq \vep \} \]
defined by the total boundary defining function $\sigma$.  This section is devoted to the proof of the following uniform Schauder estimate:
\begin{prop} \label{prop:boundary-schauder}
Let $\omega \in C^{k,\alpha}(M_{\vep};E)$ be a solution to $(\Delta + K) \omega = \tau$ in $M_{\vep}$, $\omega|_{\partial M_{\vep}} = 0$.  Then for $\vep$ sufficiently small,
\[ \| \omega \|_{k,\alpha; M_{\vep}} \leq C ( \|\tau\|_{k-2,\alpha;M_{\vep}} + \|\omega\|_{0,0;M_{\vep}} ),\]
where $C = C(n,\alpha,k,K)$ is independent of $\vep$.
\end{prop}
\begin{proof}
In order to establish this result, we follow the strategy used by Cheng-Yau \cite{CY} and Graham-Lee \cite{GrahamLee}.  It suffices to estimate the H\"older norm of $u$ locally in small balls of a uniform size both in the interior and on the boundary, where the metric and hence the coefficients of the operator $\Delta + K$, are appropriately controlled.  In a large compact set away from all of the boundary hypersurfaces, the estimate follows from classical elliptic theory.  We thus explain how to estimate in a cusp neighbourhood, where the injectivity radius shrinks to zero, and near $H_0$.

\

Fix a cusp neighbourhood $\mathcal{C}_p$ where $f$ is the rank of the cusp, as in \S \ref{s:gfhyp}.  The fact that the injectivity radius is zero is due to the side identifications of the polygon $K$ using $\Gamma$.  Thus instead of working directly in the cusp neighbourhood, we will ``unroll'' the cusp and work on the universal cover.  Recall that $M_{\vep}$ is locally defined near the cusp relative to the coordinates $(u, v^i, z^j)$ introduced before equation \eqref{eqn:hyperbolic-uvcoords} by the inequality $ u \geq \vep $ (in these coordinates the cusp is located at $u = 0, v = 0$).  Pulling back the tensor $\omega:M \to E$ to this region amounts to considering a periodic extension in the $z^j$ variables.  As the definition of H\"older spaces essentially measure local quantities and our definition of H\"older spaces are defined geometrically, the H\"older norm of the pullback to the universal cover coincides with the H\"older norm on $M$.

\

Fix a point $p_0 \in \partial M_{\vep} \cap C_{p}$.   Locally we may write $p_0 = (\vep, v^i_0, z^j_0)$.  We first consider the case where $|v^i_0| \leq C \vep$.    Consider $\bR^n = \bR_s \times \bR^{n-1-f}_p \times \bR^{f}_q$ and the euclidean reference half-ball $B_+ := \{(s,p^i,q^j): s^2 + | p |^2 + |q|^2 < 1, \; \mbox{and} \; s \geq 0 \}$, with $B'$ denoting the half-ball with radius $1/2$.  Define a map
\[ \Psi_{p_0}: B_+ \to \Psi_{p_0}(B_+) \]
by
\[ (u,v^i,z^j) = \Psi_{p_0}(s, p^i, q^j) = \left( \vep e^{s}, v_0^i + \vep p^i, z_0^j + \frac{1}{\vep} q^j \right). \] 
The hyperbolic metric of \eqref{eqn:hyperbolic-uvcoords} pulls back via $\Psi_{p_0}$ to
\[ \Psi^*_{p_0} h = ds^2 + e^{-2s} \sum (dp^i)^2 + e^{-2s} \left( e^{2s} +  \left|\frac{1}{\vep} v_0 + p\right|^2 \right)^2 \sum (dq^j)^2.\]
At least for $|v^i_0| \leq C \vep$, the coefficients of any number of derivatives of this metric are uniformly bounded on $B_{+}$.  Moreover, there is a uniform positive lower bound on the eigenvalues of the inverse of $\Psi^*_{p_0} h$.  Thus the Laplacians associated to this metric are uniformly elliptic with uniformly (in $\vep$) H\"older continuous coefficients.

\

Now pulling back $(\Delta + K) \omega = \tau$ to $B_+$, we may apply a standard local boundary Schauder estimate \cite[Corollary 6.7]{GilbargTrudinger} to obtain an estimate of the form
\[ \| \Psi^*_{p_0} \omega \|_{k,\alpha;B_+ \cap B'} \leq C ( \| \Psi^*_{p_0} \tau\|_{k-2,\alpha;B_+} + \| \Psi^*_{p_0} \omega \|_{0,0;B_+}), \]
where $C$ is independent of $\vep$ and $p_0$.  Note that the H\"older spaces referenced here are with respect to the background euclidean metric in $(s,p,q)$-coordinates. However it is straightforward to check that this implies that an estimate with H\"older spaces with respect to $\Psi^* h$ holds with a different constant since the $\Psi^* h$ and the euclidean metric and all derivatives are uniformly comparable in this chart.  We conclude that
\[ \| \omega \|_{k,\alpha; \Psi_{p_0}(B') } \leq C ( \|\tau\|_{k-2,\alpha;M_{\vep}} + \|\omega\|_{0,0;M_{\vep}} ),\]
which yields the required local estimate when $p_0$ is close to the cusp (as quantified by the inequality $|v^i_0| \leq C \vep$).

\

We now consider the case where $p_0 = (\vep, v^i_0, z^j_0) \in \partial M_{\vep} \cap C_{p}$ and  $\vep < |v^i_0| < 1$.  A slightly different rescaling is required.  Consider $\bR^n = \bR_s \times \bR^{n-1-f}_p \times \bR^{f}_q$ and the euclidean reference half-ball $B_+ := \{(s,p^i,q^j): s^2 + | p |^2 + |q|^2 < 1, \; \mbox{and} \; s \geq 0 \}$, with $B'$ denoting the half-ball with radius $1/2$.  Define the map
\[ \Psi_{p_0}: B_+ \xrightarrow{\phantom{quebecquebec}} \Psi_{p_0}(B_+) \]
by
\[ (u,v^i,z^j) = \Psi_{p_0}(s, p^i, q^j) = \left( \vep e^{s}, v_0^i + \vep p^i, z_0^j + \frac{\vep}{\vep^2 + |v_0|^2} q^j \right). \] 
The hyperbolic metric of \eqref{eqn:hyperbolic-uvcoords} pulls back via $\Psi_{p_0}$ to
\[ \Psi^*_p h = ds^2 + e^{-2s} \sum (dp^i)^2 + e^{-2s} \frac{\left( e^{2s} + \frac{|v_0|^2}{\vep^2} + 2 \frac{v_0}{\vep} \cdot p + |p|^2 \right)^2}{\left(1 + \frac{|v_0|^2}{\vep^2}\right)^2 }\sum (dq^j)^2.\]
Once more, the coefficients of any number of derivatives of this metric are uniformly bounded on $B_{+}$, independently of $\vep$, and inverse metric has eigenvalues uniformly bounded from below.  Thus the Laplacians associated to this metric are uniformly elliptic with uniformly (in $\vep$) H\"older continuous coefficients, as we obtain estimates on half-balls exactly as before.

\

This concludes the required estimate near the cusp ends.  The boundary estimates further away from the cusp but near the boundary at infinity $H_0$ may be obtained in a manner similar to that of \cite{GrahamLee}.  In the neighbourhood $\sC_0$, the hyperbolic metric is given by equation \eqref{mc.3}.  If we introduce local coordinates at a point $U$ on $H_0$, $(\rho, v^i)$, the metric may be written
\[ h = \frac{d\rho^2 + h_U( \rho, v )_{ij} dv^i dv^j}{\rho^2}. \]

We now consider $p_0 = (\vep, v^i_0) \in \partial M_{\vep} \cap C_{0}$.  Consider $\bR^n = \bR_s \times \bR^{n-1}_p$ and the euclidean reference half-ball $B_+ := \{(s,p^i): s^2 + | p |^2 < 1, \; \mbox{and} \; s \geq 0 \}$, with $B'$ denoting the half-ball with radius $1/2$.  Define a map
\[ \Psi_{p_0}: B_+ \xrightarrow{\phantom{quebecquebec}} \Psi_{p_0}(B_+) \]
by
\[ (u,v^i) = \Psi_{p_0}(s, p^i) = \left( \vep e^{s}, v_0^i + \vep p^i \right). \] 
The hyperbolic metric of \eqref{eqn:hyperbolic-uvcoords} pulls back via $\Psi_{p_0}$ to
\[ \Psi^*_p h = ds^2 + e^{-2s} h_U\left( \vep e^{s}, v_0 + \vep p \right)_{ij} dp^i dp^j.\]
Now the required Schauder estimates follow exactly as above.
\end{proof}

\section{Linear theory: isomorphism theorems for $\Delta^h + K$}
\label{s:linear}

In this section we state and prove an isomorphism theorem for Laplace operators of the form $\Delta^h + K$, where $K$ is a constant and $\Delta^h = \nabla^* \nabla $ is the Laplacian with respect to a geometrically finite hyperbolic metric.  This isomorphism theorem depends on a certain asymptotic estimate for powers of the total defining function.  We then study the asymptotic estimates required to invert these Laplacians.  Finally, we state all of the conditions required to invert the specific operators given in \eqref{linearize-scrQ} that we need in the application to Einstein metrics.

\

To begin, suppose that either $E = \Sigma^2_0(M)$ or $E = M\times\bR$.  We wish to prove an isomorphism theorem for
\[ \Delta^h + K: C^{k+2,\alpha}_{\mu}(M;E) \to C^{k,\alpha}_{\mu}(M;E). \]
We remind the reader of the exhaustion $M_{\vep}$ of $M$ defined in equation \eqref{eqn:exhaustion}, and for which the boundary Schauder estimates of \S \ref{s:schauder} apply.  The value of the weight $\mu \in \bR^{1+n_c}$ will be chosen so that the following two conditions are satisfied:
\begin{enumerate}
\item[(H1)] $ C^{0}_{\mu}(M;E)\cap \ker (\Delta^h+K)= \{0\}$,
\item[(H2)] The \textbf{asymptotic estimate holds for weight} $\mathbf{\mu}$, i.e. there exists a compact subset $M_c \Subset M$ and there exists $\delta > 0$ so that in $M \setminus M_c$, 
\[ (\Delta^h + K) \sigma^{\mu} \geq \delta \sigma^{\mu}. \]
\end{enumerate}

We will additionally assume that 

\begin{enumerate}
\item[(H3)]  For $\epsilon>0$ sufficiently small, the Dirichlet problem 
\begin{equation*} 
  \begin{cases}
    (\Delta^h + K) u = f, & \text{in} \; M_{\epsilon}\setminus \partial M_{\epsilon}.\\
    u = 0, & \text{on} \; \partial M_{\epsilon}.
  \end{cases}
\end{equation*} 
has a unique solution for all $f\in L^2(M_{\epsilon};E)$.
\end{enumerate}
These assumptions yields the following isomorphism theorem.  

\begin{theorem} \label{thm:P1-iso}
Let $K$ be  constant and suppose that $\mu$ is a weight for which hypotheses (H1), (H2) and (H3) hold.  Then
\[ P_1 = \Delta^{h} + K: C^{k+2,\alpha}_{\mu}(M;E) \longrightarrow C^{k,\alpha}_{\mu}(M;E) \]
is an isomorphism for all $k\in \bN_0$.
\end{theorem}

\begin{proof}
Since $C^{k+2,\alpha}_{\mu}(M;E)\subset C^0_{\mu}(M;E)$, we see by (H1), that the operator $P_1$ is injective. To show that $P_1$ is surjective with bounded inverse, fix $f \in C^{k,\alpha}_{\mu}(M;E)$.  Now recalling the exhaustion $M_{\vep}$ defined at the beginning of \S \ref{s:schauder}, consider for $j$ sufficiently large the Dirichlet problem on $M_{1/j}$:
\begin{equation} \label{local-dirichlet} \tag{$*$}
  \begin{cases}
    P_1 u = f, & \text{in} \; M_{1/j}.\\
    u = 0, & \text{on} \; \partial M_{1/j}.
  \end{cases}
\end{equation} 
For $j$ sufficiently large, $M_{1/j}$ are smooth manifolds with boundary, and by (H3), there is a unique solution $u$.  By classical elliptic PDE theory, this unique solution $u$ is in $C^{k+2,\alpha}_{\mu}(M_{1/j};E)$. 

\

Moreover, by Schauder theory, the inverse map to \eqref{local-dirichlet} is bounded, i.e. there exists $C_j > 0$ where
\begin{equation} \label{eqn:uniform-est-P1}
 \| u \|_{k+2,\alpha,\mu; M_{1/j} } \leq C_j \| P_1 u \|_{k,\alpha,\mu; M_{1/j}}. \end{equation}

\

\underline{\textbf{Claim:}} The constant $C_j$ may be taken independently of $j$.  If this is not the case, then passing to a subsequence if necessary, we can suppose that there is a sequence $u_j \in C^{k+2,\alpha}_{\mu}(M_{1/j})$ with $\| u_j\|_{k+2,\alpha,\mu;M_{1/j}}~=~1$ but 
\[ \| P_1 u_j \|_{k,\alpha,\mu;M_{1/j}} \to 0 \quad \mbox{as} \; j\to \infty. \]

\

Applying the boundary Schauder estimate, Proposition \ref{prop:boundary-schauder}, we  show that for $j$ sufficiently large the maximum value of $|u_j|/\sigma^{\mu}$ is uniformly bounded from below.  Indeed for some constant $C$ independent of $j$,
\[ 1 \equiv \| u_j \|_{k+2,\alpha,\mu;M_{1/j}} \leq C \left( \| P_1 u_j \|_{k,\alpha;\mu;M_{1/j}} + \| u_j \|_{0,0,\mu;M_{1/j}} \right). \]
Thus for $j$ sufficiently large, since $\| P_1 u_j \|_{0,\alpha,\mu;M_{1/j}} \to 0$ we find that 
\begin{equation}
\frac{1}{C} \stackrel{<}{\sim} \| u_j\|_{0,0,\mu;M_{1/j}} = \| u_j / \sigma^{\mu} \|_{0,0; M_{1/j}}, 
\label{lbmax.1}\end{equation}
which shows that the maximum value of $|u_j| / \sigma^{\mu}$ is uniformly bounded from below in $j$ on $M_{1/j}$.

\

Using the (classical) maximum principle, we will now show that the location of the maxima of $|u_j|/\sigma^{\mu}$ cannot drift to infinity.  Choose a large compact subset $M_c \Subset M$ so that on $M \setminus M_c$  the asymptotic estimate (H2) holds for weight $\mu$.  We then assume that $j$ is chosen so large that $M_c \Subset M_{1/j}$.

\

Now consider the function $|u_j| / \sigma^{\mu}$ on $\overline{M \setminus M_c} \cap M_{1/j} $.  Let $p_j$ be a point realizing the maximum value, i.e. 
\[ \frac{|u_j|}{ \sigma^{\mu}}(p_j) = \max_{\overline{M \setminus M_c} \cap \overline{ M_{1/j} }} \frac{|u_j|}{ \sigma^{\mu}}.\]
Note that $p_j \not\in \partial M_{1/j}$ since $u_j$ vanishes there.

\

If instead the maximum is attained in the interior of $M_{1/j} \setminus M_c$, then note that the argument in \cite[Proposition 3.8]{GrahamLee} implies that
\begin{equation} \label{eqn:loc-maxprinc}
 \delta \frac{|u_j|}{\sigma^{\mu}} (p_j) \leq \frac{|P_1 u_j|}{\sigma^{\mu}} (p_j).
\end{equation}
To see this estimate, one simply computes as in \cite{CY, GrahamLee} that for some first order operator $V$ depending on $\sigma$,
\[ (\Delta^h-V+\delta)\frac{|u_j|}{\sigma^{\mu}}\le \left(\Delta^h - V + K + \frac{\Delta^h \sigma^{\mu}}{\sigma^{\mu}} \right) \frac{|u_j|}{\sigma^{\mu}} \le \frac{\langle(\Delta^h + K) u_j,u_j\rangle}{\sigma^{\mu}|u_j|}. \]
Evaluating $\frac{|u_j|}{\sigma^{\mu}}$ at the maximum $p_j$ and using the asymptotic estimate allows us to obtain estimate \eqref{eqn:loc-maxprinc}.  

\

Hence we see that
\[ \| u_j \|_{0,0,\mu; M_{1/j} \cap M \setminus M_c} = \frac{|u_j|}{\sigma^{\mu}} (p_j) \leq \frac{1}{\delta} \frac{|P_1 u_j|}{\sigma^{\mu}} (p_j) \leq \frac{1}{\delta}\| P_1 u_j \|_{0,\mu;M_{1/j} \cap M \setminus M_c} \leq \frac{1}{\delta} \| P_1 u_j \|_{k,\alpha,\mu;M_{1/j}}.\]
Since by assumption $\| P_1 u_j \|_{k,\alpha,\mu;M_{1/j}}\to 0$ as $j\to \infty$, it follows that $\frac{|u_j|}{\sigma^{\mu}} (p_j)$ is very small when $j$ is large.  By our earlier uniform lower bound \eqref{lbmax.1}, this means that for $j$ sufficiently large, the maximum value of $|u_j|/\sigma^{\mu}$ over $M_{1/j}$ cannot  occur in the interior of $\overline{M\setminus M_c}\cap M_{1/j}$, that is, it must occur at a point $q_j$ on the compact set $M_c$.

\

The argument now finishes in a standard way.  It is possible to choose a subsequence of $q_j$ and of the original sequence $u_j$ so that $q_{j_i} \to q$ and $u_{j_i}$ converges on compact subsets in $C^{k+2}$ to a function $u\in C^{k+2,\alpha}_{\mu}(M;E)$.  Additionally $u \neq 0$ since $\left(|u|/\sigma^{\mu}\right)(q)$ is at least $1/c$.  However, $P_1 u = \lim_{i \to \infty} P_1 u_{j_i}$, and since $\| P_1 u_{j_i} \|_{k,\mu;M_{1/j_i}} \to 0$, it follows $\| P_1 u \|_{k,\mu;M} = 0$ so that $P_1 u = 0$.  But then $u = 0$ by injectivity, a contradiction.  This establishes the claim and the uniform estimate 
\begin{equation} \label{ue.1}
 \| u \|_{k+2,\alpha,\mu; M_{1/j} } \leq C \| P_1 u \|_{k,\alpha,\mu; M_{1/j}} 
 \end{equation}
for a constant $C>0$ not depending on $j$.  

\

The remainder of the proof now finishes as in \cite[Proposition 3.7]{GrahamLee}, and we repeat the details for completeness.  Let $v_j \in C^{k+2,\alpha}_{\mu}(M_{1/j};E)$ be the unique solution to \eqref{local-dirichlet} in $M_{1/j}$.  The uniform estimate \eqref{ue.1} implies 
\begin{equation} \label{ue.2}
 \| v_j \|_{k+2,\alpha,\mu; M_{1/j} } \leq C \| f \|_{k,\alpha,\mu; M_{1/j}} \leq C \| f \|_{k,\alpha,\mu; M}.
 \end{equation}
The Arzela-Ascoli theorem implies that a subsequence of $v_j$ converges in $(k+2,0,\mu)$-norm on $\overline{M_{1/j}}$, and a diagonal subsequence argument provides a limit $v \in C^{k+2}_{\mu}(M;E)$ and further subsequence with $v_{\ell} \to v$ in  $C^{k+2}_{\mu}(M_{1/j};E)$ for any $j$.  Thus $(\Delta^h + K) v =   f$, and passing to limits in \eqref{ue.2} shows $v \in C^{k+2,\alpha}_{\mu}(M;E)$ and completes the proof of both surjectivity and boundedness of the inverse.
\end{proof}

The previous theorem can be applied to the scalar Laplacian as follows.
\begin{cor}
Given a non-negative constant $K$, suppose that $\mu$ is weight such that (H2) holds and $C^{0,0}_{\mu}(M)\subset L^2(M)$.    Then
\[ P_1 = \Delta^{h} + K: C^{k+2,\alpha}_{\mu}(M) \longrightarrow C^{k,\alpha}_{\mu}(M) \]
is an isomorphism.
\label{sclp.1}\end{cor}
\begin{proof}
A simple integration by parts shows that $P_1$ is injective on any compact domain with smooth boundary if we impose Dirichlet boundary conditions.  Since the operator is formally self-adjoint, the Fredholm alternative shows that the Dirichlet problem always has a unique solution, so condition (H3) of Theorem~\ref{thm:P1-iso} is satisfied.    Additionally, if $u \in C^{\infty}_{c}(M)$ satisfies $P_1 u = 0$ then by  integrating by parts again,
\[ 0 = \int_M u P_1 u \; dv_g = \int_M |\nabla u|^2 + K u^2 \; dv_g, \]
from which it follows $u = 0$ since $K \geq 0$.  Thus $P_1$ is injective on any space for which compactly supported smooth functions are dense and in particular on $L^2(M)$.  Since we are assuming $C^0_{\mu}(M)\subset L^2(M)$, this means that condition (H1) of Theorem~\ref{thm:P1-iso} also holds.  We can therefore apply Theorem~\ref{thm:P1-iso} to obtain the result.
\end{proof}

The previous corollary extends easily to the line bundle $\sG$ of smooth multiples of the hyperbolic metric since $(\Delta^h + K)[ u h ] = (\Delta^h + K)[ u ] \cdot h$.  Note that there is no shift in weights in our conventions.
\begin{cor}
Given a non-negative constant $K$, suppose that $\mu$ is weight for which hypothesis (H2) holds and $C^{0,0}_{\mu}(M;\sG)\subset L^2(M;\sG)$.  Then
\[ P_1 = \Delta^{h} + K: C^{k+2,\alpha}_{\mu}(M;\sG) \longrightarrow C^{k,\alpha}_{\mu}(M;\sG) \]
is an isomorphism.
\end{cor}

Next we consider an isomorphism theorem for the Laplacian acting on trace-free symmetric $2$-tensors.  

\begin{cor} \label{thm:P2-iso}
Given a constant $K > -n$, suppose that $\mu$ is a weight for which (H2) holds and $C^{0}_{\mu}(M;\Sigma^2_0(M))\subset L^2(M;\Sigma^2_0(M))$.  Then
\[ P_2 = \Delta^{h} +K: C^{k+2,\alpha}_{\mu}(M;\Sigma^2_0(M)) \longrightarrow C^{k,\alpha}_{\mu}(M;\Sigma^2_0(M)) \]
is an isomorphism.
\end{cor}

\begin{proof}
To apply  Theorem~\ref{thm:P1-iso}, we need to show that hypotheses (H1) and (H3) hold.  As we now explain, both of these follow from Koiso's trick used in slightly different settings.  To facilitate, let $(N^n,h)$ be a hyperbolic Riemannian $n$-manifold taken as follows:
\begin{itemize}

\item For hypothesis (H1), we take $(N,h) = (M,h)$, which is a complete noncompact  manifold.  Note that $C^{\infty}_c(N;\Sigma^2_0(N))$ is dense in $L^2(N;\Sigma^2_0(N))$.

\item For hypothesis (H3), we take $(N,h) = (M_{\vep},h)$, which is a compact manifold with boundary.  Here $C^{\infty}_c(N;\Sigma^2_0(N))$ is dense in $H^{2,2}_0(N;\Sigma^2_0(N))$ and encodes the Dirichlet boundary condition.
\end{itemize}

We now review Koiso's trick for the convenience of the reader \cite[Proof of Theorem A]{Lee}.  For $(N, h)$ one of the manifolds above, recall that  the Lichnerowicz Laplacian of $h$ acting on symmetric $2$-tensors is given by
\[ \Delta_L u = \nabla^* \nabla u + 2 \stackrel{\circ}{\Rc} u - 2 \stackrel{\circ}{\Rm} u \]
where $[\stackrel{\circ}{\Rc} u ]_{ij} = \frac{1}{2} ( R_{ik} {u_j}^k + R_{kj} {u_i}^k)$ and $[\stackrel{\circ}{\Rm} u ]_{ij} = R_{iklj} u^{kl}$.  For a hyperbolic metric we insert $\Rc(h) = -(n-1) h$ and $\Rm(h)_{iklj} = -( h_{ij} h_{kl} - h_{il} h_{kj} )$ and then
\[ \Delta_L^h u = \nabla^* \nabla u -2(n-1) u + 2 (\tr^h u) g - 2u =  \nabla^* \nabla u -2 n u + 2 (\tr^h u) h. \]

\

Suppose that $\lambda$ is a constant.  In the calculations that follow we take $u \in C^{\infty}_c(N;\Sigma^2_0(N))$ for either choice of $N$.  In both cases there are no boundary terms that arise from the integration by parts in the following calculations:
\begin{align} \label{equation-naive-IBP}
(u, (\Delta_L^h + \lambda)u) &= \int_N \inn{u}{\nabla^* \nabla u + (\lambda-2 n) u + 2 (\tr^h u) h}{h} dV_h \nonumber \\
&= \| \nabla u \|^2_{L^2} + (\lambda-2n) \|u\|^2_{L^2} + 2 \| \tr^h u \|^2_{L^2} \\
&\geq (\lambda-2n) \|u\|^2_{L^2}. \nonumber
\end{align}
To improve the constant, define a $3$-tensor by
\[ T_{ijk} := \nabla_k u_{ij} - \nabla_i u_{jk}. \]
In the following computation using indices, we use the convention that we sum over repeated indices and omit upper/lower position, and indices following a comma denote covariant differentiation.  Once more there are no boundary terms in the following integration by parts:
\begin{align*}
\|T\|^2_{L^2} &= \int_N T_{ijk} T_{ijk} \; dV_h \\
&= \int_N (u_{ij,k} - u_{jk,i})(u_{ij,k} - u_{jk,i}) \; dV_h \\
&= \int_N 2|\nabla u|^2 - 2 u_{ij,k} u_{jk,i} \; dV_h \\
&= \int_N 2|\nabla u|^2 + 2 u_{ij,ki} u_{jk} \; dV_h \\
&= \int_N 2|\nabla u|^2 + 2 (u_{ij,ik} + R_{kiis} u_{sj} + R_{kijs} u_{is} ) u_{jk} \; dV_h \\
&= \int_N 2|\nabla u|^2 + 2 u_{ij,ik} u_{jk} + 2R_{ks} u_{sj} u_{jk} + 2R_{kijs} u_{is} u_{jk} \; dV_h \\
&= \int_N 2|\nabla u|^2 - 2 u_{ij,i} u_{jk,k} -2(n-1) |u|^2 -2 (h_{ij}h_{ks} - h_{kj} h_{is} ) u_{is} u_{jk} \; dV_h \\
&= \int_N 2|\nabla u|^2 - 2 u_{ij,i} u_{jk,k} -2n |u|^2 + 2(\tr^h u)^2 \; dV_h.
\end{align*}
We conclude that 
\[ \| \nabla u\|^2_{L^2} = \frac{1}{2} \|T\|^2_{L^2} + \| \mathrm{div}^h u \|^2_{L^2} - \|\tr^h u\|^2_{L^2} + n \|u\|^2_{L^2}. \]
Combining this equation with the calculation of equation \eqref{equation-naive-IBP}, we find that 
\begin{align*}
(u, (\Delta_L^h + \lambda)u) &= \| \nabla u \|^2_{L^2} + (\lambda-2n) \|u\|^2_{L^2} + 2 \| \tr^h u \|^2_{L^2} \\
&= \frac{1}{2} \|T\|^2_{L^2} + \| \mathrm{div}^h u \|^2_{L^2} + \|\tr^h u\|^2_{L^2} + (\lambda-n) \|u\|^2_{L^2}\\
&\geq (\lambda-n) \|u\|^2_{L^2}.
\end{align*}

In our analysis, we are interested in $P_2 = \nabla^* \nabla + K$.  Note that
\begin{align*}
(u, P_2 u) &= (u, (\nabla^* \nabla + K) u ) \\
&= (u, \Delta_L u - 2 \stackrel{\circ}{\Rc} u + 2 \stackrel{\circ}{\Rm} u +Ku) \\
&= (u, \Delta_L u + 2(n-1) u - 2(\tr^h u) h + 2 u + K u) \\
&= (u, \Delta_L u + (2n+K) u)\\
&\geq ( 2n+K - n) \|u\|^2_{L^2}\\
&\geq ( n+K ) \|u\|^2_{L^2}.
\end{align*}

Thus we obtain the estimate $\|P_2 u\|_{L^2} \geq (n+K) \|u\|_{L^2}$ in both of the choices for $N$.  Hypothesis (H1) now follows as we take values of $\mu$ so that $C^{k+2,\alpha}_{\mu}$ embeds into $L^2$.  Hypothesis (H3) follows from the Fredholm alternative since solutions to the homogeneous Dirichlet problem are unique.  Thus we may apply Theorem~\ref{thm:P1-iso}.
\end{proof}

\subsection{Asymptotic estimates}

We now consider the asymptotic estimate (H2) above.  We separate into cases by estimating near intermediate rank cusps, near maximal rank cusps and near $H_0$.

\

\subsubsection*{Intermediate rank cusps}
Consider an intermediate rank cusp neighbourhood, where the rank of the cusp is $f < n-1$ and set  $b := n-1-f$.  We write the metric of \eqref{mc.1} in a slightly different way to facilitate computation.  Recalling that $d\phi^2$ is the round metric on $\bS^{b}$, let $\theta_0$ denote the Riemannian distance from the north pole.  Thus $\rho = \cos(\theta_0)$.  Then $d\phi^2 = d\theta_0^2 + \sin^2(\theta_0) d \theta_{\alpha}^2$, where $d\theta_{\alpha}^2$ is the round metric on $\bS^{b-1}$.  The metric \eqref{mc.1} can then be written
\begin{equation} \label{m1}
h = \frac{dr^2}{r^2 \cos^2(\theta_0)} + \frac{d \theta_0^2 +  \sin^2(\theta_0) d \theta_{\alpha}^2}{\cos^2(\theta_0)} + \frac{r^2}{\cos^2(\theta_0)} dw^2,
\end{equation}
where $dw^2$ is a flat metric.

\

Recall that for a generic metric $g$, the coordinate expression for the Laplacian on functions is given by 
\[ \Delta^g u =   \frac{-1}{\sqrt{ |\det g|}} \partial_i ( \sqrt{ |\det g|} g^{ij} \partial_j u). \]

For the hyperbolic metric \eqref{m1}, note that
\[ \det h = \frac{r^{2(f-1)} \sin(\theta_0)^{2(b-1)}}{\cos(\theta_0)^{2n}}  \det(d\theta_{\alpha}^2) \det(dw^2).\]

A straightforward computation shows
\begin{align} 
-\Delta^h u &= \cos^2(\theta_0) r^2 \partial_r^2 u + (f+1) \cos^2(\theta_0) r \partial_r u \nonumber \\
&+ \cos^2(\theta_0) \partial_{\theta_0}^2 u + (n-2) \sin(\theta_0) \cos(\theta_0) \partial_{\theta_0} u + (b-1) \frac{\cos^3(\theta_0)}{\sin(\theta_0)} \partial_{\theta_0} u + \frac{\cos^2(\theta_0)}{\sin^2(\theta_0)} \Delta^{\bS^{b-1}} u \\
& + \frac{(\cos \theta_0)^{2}}{r^{2}}\Delta^{dw^2} u \nonumber.
\end{align}

\

Now, applying the Laplacian plus a constant $K$ to the function $u = r^{\mu} (\cos \theta_0)^{\nu}$, $\mu, \nu \in \bR$, we obtain
\[ (\Delta^h + K)( r^{\mu} (\cos \theta_0)^{\nu} ) = \left[ K - (\mu^2 + f \mu -  b \nu)\cos(\theta_0)^2 - ( \nu(\nu - (n-1) )\sin(\theta_0)^2  \right] r^{\mu} (\cos \theta_0)^{\nu}. \]

Thus a simple condition ensuring that (H2) holds in this asymptotic end for some $\delta > 0$ is that both
\begin{equation} \label{a.01}  K - (\mu^2 + f \mu -  b \nu) > \delta \; \mbox{and} \; K - ( \nu(\nu - (n-1) ) > \delta \end{equation}
hold simultaneously.

\subsubsection*{Maximal rank cusps} Now consider a cusp neighbourhood of a maximal rank cusp, with metric given by \eqref{mc.2}.  The reader may check that applied to $u = r^{\mu}$, one obtains
\[ (\Delta^h + K) r^{\mu} = (K - \mu(\mu + (n-1)) ) r^{\mu},\]
so that (H2) holds in this asymptotic end for some $\delta > 0$ if
\begin{equation} \label{a.02} K - \mu(\mu + (n-1)) > 0. 
\end{equation}

\subsubsection*{Asymptotic estimate near $H_0$}  Near a point of $H_0$, the metric is given by \eqref{mc.3}.  A calculation shows that applied to $u = \rho^{\nu}$, one obtains
\[ (\Delta^h + K) \rho^{\nu} = (K - \nu(\nu - (n-1))) \rho^{\nu}, \]
so that (H2) holds near $H_0$ for some $\delta > 0$ if
\begin{equation} \label{a.03}
 K - \nu(\nu - (n-1)) > 0. 
\end{equation}

\subsection{Specialization to $L$} In our application to Einstein metrics of the next section, we need to invert the operator $L$ given in equation \eqref{linearize-scrQ}.  In order to apply Corollaries \ref{sclp.1}--\ref{thm:P2-iso} we must require each weight in a weight vector $\mu \in \bR^{1+n_c}$ to be larger than the $L^2$ cutoff, as given in Lemma \ref{l2-cut}:
\[ \mu_0 > \frac{n-1}{2}, \; \mbox{and} \; \mu_j > -\frac{f_j}{2}. \]
where $f_j$ is the rank of the $j$-th cusp.   In fact we are forced to choose $\mu_j > 0$ in order to preserve the asymptotic class of the metric. 

\

In order to satisfy estimates \eqref{a.01}, \eqref{a.02}, and \eqref{a.03} simultaneously for both $K = 2(n-1)$ and $K=-2$, it suffices to choose weights for $K = -2$.  Unfortunately a quick analysis of \eqref{a.02} shows that there are no positive weights for which we may obtain an isomorphism theorem.  Thus we cannot apply our method to cusps of maximal rank.

\

Near $H_0$ we must have
\[ \mu_0 > \frac{n-1}{2} \; \mbox{and} \; \mu_0 (\mu_0 - (n-1) ) < -2, \]
which forces  
\begin{equation} \label{eqn:mu0-const} \mu_0 \in \left( \frac{n-1}{2}, \frac{n-1}{2} + \frac{\sqrt{(n-1)^2-8}}{2}\right). 
\end{equation}

Near a cusp of intermediate rank, we must take $\mu_i$ slightly positive so that
\[ -2 > (\mu_i^2 + f_i \mu_i - b_i \mu_0) = (\mu_i^2 + f_i \mu_i - (n-1-f_i) \mu_0). \]

We now give two examples of weights for which the asymptotic estimate holds.
\begin{prop} \label{prop:admissible-weight} If $n > 4$ and if $(M^n,h)$ is a geometrically finite hyperbolic metric with $n_c$ cusps of intermediate rank $f_j$, $j = 1, \ldots, n_c$, then
the asymptotic estimate holds for $K=-2$ with weights
\[ \mu_0 = n-2, \mu_i = \frac{1}{n-2}.\]
\end{prop}
\begin{proof}
An easy calculation shows that $\mu_0 = n-2$ satisfies $\mu_0(\mu_0 - (n-1)) < -2$ when $n > 4$, and further, $\mu_0 > \frac{1}{2}(n-1)$.  Recall now that intermediate rank cusps entail that $1 \leq f_i \leq n-2 < n-1$.  So for $\mu_i = \frac{1}{n-2}$ we have
\begin{align*}
\mu_i^2 + f_i \mu_i &= \frac{1}{(n-2)^2} + f_i \frac{1}{n-2}\\
&\leq \frac{1}{(n-2)^2} + (n-2) \frac{1}{n-2}\\
&< 2 < n \\
&\leq -2 + (n-2) \\
&\leq -2 + (n-1 -f_i)(n-2),
\end{align*}
which gives the required estimate.
\end{proof}

\begin{prop} \label{prop:admissible-weight2} For $n=4$, the asymptotic estimate for a geometrically finite hyperbolic metric $(M^n,h)$ with $n_c$ cusps only holds if all the cusps are of rank $f_i = 1$.  Any value of $\mu_0 \in (3/2,2)$ and $\mu_i$ sufficiently small (depending on $\mu_0$) will satisfy the asymptotic estimate.
\end{prop}
\begin{proof}
First, note that when $n=4$, the constraint of \eqref{eqn:mu0-const} implies that $\mu_0 \in (3/2,2)$.  

\

Since the cusps are of intermediate rank, we find that there are only two cases to check, when $f_i = 1$ and when $f_i =2$.  When $f_i = 2$, we see that the constraint on $\mu_i$ reads
\[ \mu_i^2 + 2 \mu_i < \mu_0 - 2, \]
and the right hand side is always nonpositive for admissible values of $\mu_0$.  Thus no positive weights $\mu_i$ are admissible when $f_i = 2$.

\

When $f_i = 1$, we find
\[ \mu_i^2 + \mu_i < 2\mu_0 - 2, \]
and the right hand side is now always positive for admissible values of $\mu_0$.  Thus the constraint for $\mu_i$ can always be satisfied for $\mu_i$ sufficiently small and positive.
\end{proof}

\section{Einstein metrics near a geometrically finite hyperbolic metric}
\label{s:proof}

In this section we prove the existence of Einstein perturbations of a geometrically finite hyperbolic metric of Theorem \ref{thm:main}.

\

Fix a hyperbolic metric $(M,h)$ with $n_c$ intermediate rank cusps  and consider the compactification of $M$ to the manifold with corners $\Mdbar$ of \S \ref{s:gfhyp}.  Rescale the metric $h$ by the square of the defining function for $H_0$, to obtain a partial conformal compactification
\[ \hbar := \rho^2 h. \]
We then restrict to the hypersurface $H_0$,
\[ \hat{h}: = (\rho^2 h)|_{TH_0}, \]
obtaining a metric analogous to the conformal infinity of a conformally compact metric, except that $H_0$ is a noncompact manifold with boundary and $\hat{h}$ is a foliated cusp metric.

\

We will be interested in compactly supported perturbations of $\hat{h}$, i.e. perturbations away from the cusp faces.  To this end, let $U$ be an open set of $H_0\setminus \partial H_0$ with closure also contained in the interior of $H_0$, so that the inward normal exponential map of $\overline{h} = \rho^2 h$ is a diffeomorphism from a small product neighbourhood of $U \times [0,\delta)$ to a collar neighbourhood of $U$ in $\Mdbar$ away from the cusp hypersurfaces.  In what follows we implicitly use this exponential map to identify this neighbourhood with the product $U \times [0,\delta)$.  If $y^{\alpha}$ are arbitrary coordinates on $U$, we extend them into $M$ by declaring them to be constant along the integral curves of $\nabla^{\hbar} \rho$.  In this neighbourhood, the hyperbolic metric is of the form
\[  h = \frac{d \rho^2 + h_{U}(\rho)}{\rho^2} \]
for some family of metrics $h_U$ on $U$ smoothly parametrized by $\rho$.  To prove our main result, we first need to construct a approximate solution to the equation $Q(g,t)=0$.  
This is achieved by adapting the construction of asymptotic solutions to $Q = 0$ from Theorem 2.11 of \cite{GrahamLee}.  Note that we replace $\partial M$ with $U$, $M$ with $U \times (0,\delta)$, and $\Mbar$ with $U \times [0,\delta)$.

\

We will make use of the asymptotic expansion spaces, $A^{m}_{k,\alpha}(M)$ of Graham-Lee.  Recall that a function $f$ is in $A^m_{k,\alpha}(U \times [0,\delta) )$ if it can be written as a sum
\[ f = w_{k+m} + \rho w_{k+m-1} + \cdots + \rho^m w_k, \]
where $w_j \in C^{j,\alpha}(U \times [0,\delta))$, and we emphasize the regularity here is taken up to the boundary $\rho = 0$.  A symmetric $2$-tensor lies in $A^m_{k,\alpha}(\Sigma^2(U \times (0,\delta)))$ if its components in any coordinate system up to the boundary lie in $A^m_{k,\alpha}(U \times (0,\delta))$.  As explained in \cite{GrahamLee}, these spaces are well-defined and Banach spaces under an appropriate norm.  Moreover, the gauge-adjusted Einstein operator
\[ Q: \rho^{-2} A^{m}_{k+2,\alpha}(\Sigma^2(U \times (0,\delta))) \times \rho^{-2} A^{m}_{k+2,\alpha}(\Sigma^2(U \times (0,\delta))) \longrightarrow \rho^{-2} A^{m+2}_{k,\alpha}(\Sigma^2(U \times (0,\delta))) \]
is a smooth map.

\

We first describe an extension procedure that takes $\hat{q}\in C^{k,\alpha}_c(U;\Sigma^2(U))$ of compact support on $U$  to a perturbation of the hyperbolic metric on $M$.  Choose a non-negative $C^{\infty}(\Mdbar)$ bump function $\psi$ so that $\psi \equiv 1$ on $U \times [0,\delta/2)$ and $\psi$ is has compact support in $ (H_0\setminus \partial H_0)\times [0,\delta)$. Now extend $\hat{q}$ to a tensor $\qbar$ on $U \times [0,\delta)$ by first declaring $\nabla^{\hbar} \rho \intprod \qbar = d \rho$ and then parallel translating along the $\hbar$-geodesics normal to $H_0$.  Finally, for $\hat{g} = \hat{h}+\hat{q}$, set 
\[ E(\hat{g}) = \overline{h} + \psi \qbar. \]
We observe that $E(\hhat) = \hbar$.  The extension operator in coordinates essentially yields
\[ E(\ghat) = \hbar + \psi(\rho,y) \qbar(y)_{\alpha \beta} dy^{\alpha} dy^{\beta},\]
and we then define
\begin{align*}
T: C^{k,\alpha}(U;\Sigma^2(U)) &\longrightarrow \rho^{-2} A^0_{k,\alpha}(\Sigma^2(U \times [0,\delta))) \\
T(\ghat) &= \rho^{-2} E(\hat{g}) = h + \rho^{-2} \psi \qbar.
\end{align*}
$T$ is a smooth map of Banach spaces.  Further, the metric $T(\ghat)$ is conformally compact in $U \times [0,\delta)$ since $\rho^2 T(\ghat)$ extends to a $C^{k,\alpha}$ metric on $U \times [0,\delta)$.  

\

Now set $g_{1} = T(\ghat)$.  We first check that $g_{1}$ is a first-order solution to the gauge-adjusted Einstein equation
\[ Q(g_{1},g_{1}) = \Rc(g_{1}) + (n-1) g_{1}. \]
Note that  the gauge term vanishes when both arguments of $Q$ are identical.
Outside of the support of $\psi$ and in particular near any cusp face, the metric $g_{1}$ is identical to the hyperbolic metric and thus $Q(g_{1},g_{1}) = 0$. Inside the support of $\psi$, since the metric $g_{1}$ is conformally compact and asymptotically hyperbolic, the components of $Q(g_{1},g_{1})_{jk}$ are $O(\rho^{-1})$ relative to the coordinate system described above, and in fact $Q(g_1,g_1) \in \rho^{-1} A^1_{k-2,\alpha}$.  Thus $|Q(g_{1},g_{1})|_h = O( \rho )$, and $g_1$ is a first order solution as claimed.  Unfortunately, the fact that $Q(g_1,g_1)$ is in $\rho^{-1} A^1_{k-2,\alpha}$ is insufficient to apply the isomorphism theorem since the corresponding weighted H\"older space\footnote{The precise embedding of the $\rho^{\mu} A^{m}_{k,\alpha}$ spaces into our H\"older spaces is discussed at the end of this section.} will not embed into $L^2$. Hence the next step in the argument is to construct a finite series of metrics that solve the gauge-adjusted Einstein equation to a sufficiently high order.

\

Fixing $g_{1}$ as the background reference metric for $Q$, we now seek a higher order correction $g_{2} = g_{1} + r$ such that $Q(g_{2},g_{1})_{jk} = O(1) \; \; \left( =  O(\rho^{0}) \right)$, or
\[ |Q(g_{2},g_{1})|_h = O( \rho^2 ). \]
Applying Taylor's theorem, Graham and Lee show that for an $r = \rho^{-1} \qhat_1 \in C^{k-1,\alpha}(U \times [0,\delta))$,
\[ Q(g_{1} +  r, g_{1}) = Q(g_{1},g_{1}) + L( \rho^{-1} \qhat_1 ) + O(\rho), \]
where to leading order $L$ is the operator of $\eqref{linearize-scrQ}$.   Thus to obtain $g_{2}$ we must solve for $\qhat_1$ with 
\[ L( \rho^{-1} \qhat_1 ) = -\rho^{-1} \cdot \left\{ Q(g_{1},g_{1}) \right\}_{-1}, \]
where the notation $\{ \cdot \}_{-1}$ represents the coefficient of the $\rho^{-1}$ term in the Taylor expansion of $Q(g_1,g_1)$.  One may solve for the correction tensor $r$ above by a purely algebraic process arising through an analysis of the indicial operators of $L$ on tensors.  Further, $\qhat_1$ and therefore $g_2$ remain compactly supported within the support of $\psi$ since $Q(g_1,g_1)_{jk} = O(\rho^{-1})$ is supported within the support of $\psi$.  Thus the process remains completely local in the sense that $Q(g_2,g_1)_{jk} = 0$ outside the support of $\psi$ and $Q(g_2,g_1)_{jk} = O(1)$ within the support.  We thus obtain the expansion
\[ g_2 = g_1 + r_1 = h + \rho^{-2} \psi \qhat + \rho^{-1} \qhat_1. \]
Thus $g_2 \in \rho^{-2} A^1_{k-1,\alpha}( \Sigma^2(U \times [0,\delta) )$, and
$Q(g_2,q_1) \in \rho^{0} A^1_{k-3,\alpha}$, where we recall that we are assuming that $k > n$.

\

We may continue to improve the order of vanishing of the approximate solution inductively until the first characteristic exponent of $L$ appears. For our dimension convention, we obtain an expansion
\[ g_{n-1} = h + \rho^{-2} \psi \qhat + \rho^{-1} \qhat_1 + \cdots + \rho^{n-3} \qhat_{n-1}, \]
where $Q(g_{n-1}, g_1)_{jk} = O( \rho^{n-3} ),$ or equivalently $|Q(g_{n-1},h)|_h = O(\rho^{n-1})$.  In fact $g_{n-1} = \rho^{-2} A^{n-2}_{k-n+2,\alpha}(\Sigma^2(U \times [0,\delta)))$ and $Q(g_{n-1},g_1) \in \rho^{n-3} A^1_{k-n,\alpha}(\Sigma^2(U \times [0,\delta)))$.

\

Finally, we summarize this discussion.  The map
\begin{align*}
S: C^{k,\alpha}( U ;\Sigma^2(u)) &\longrightarrow  \rho^{-2} A^{n-2}_{k-n+2,\alpha}(\Sigma^2(U \times [0,\delta) ) )\\
S(\ghat) &= g_{n-1}.
\end{align*}
that gives
\[ |Q(S(\ghat), T(\ghat))|_h = O(\rho^{n-1}) \]
is a smooth map of Banach spaces.  The composition $\ghat \to Q( S(\ghat), T(\ghat)) \in \rho^{n-3} A^1_{k-n,\alpha}(\Sigma^2(U~\times~[0,\delta)))$ is additionally a smooth map.

\

It remains to embed the asymptotic expansion spaces into our H\"older spaces.  Recall Proposition 3.3(12) of \cite{GrahamLee} shows that
\[ A^{m}_{k,\alpha} \subset \Lambda^{0}_{k,\alpha} \]
continuously for any $m$, where $\Lambda^{0}_{k,\alpha}$ are the Graham-Lee H\"older spaces.  Consequently, 
\[ \rho^{\mu} A^{m}_{k,\alpha} \subset \Lambda^{\mu}_{k,\alpha} \]
continuously for any $m$.  Away from the cusp hypersurfaces our $C^{k,\alpha}_{\mu}$ spaces are equivalent to $\Lambda^{\mu_0}_{k,\alpha}$ except our conventions for measuring the norm of a symmetric $2$-tensor using the metric instead of the norm of components in smooth background components lead to a shift in weight by $2$ in terms of $\rho$.  Since the tensors above are equal to the hyperbolic metric outside the support of $\psi$, they have infinite order vanishing with respect to the hyperbolic metric at the cusp ends.  Thus a symmetric $2$-tensor that is a compactly supported perturbation of $h$ in $U \times [0,\delta)$ lying in $\rho^{c} \Lambda^{m}_{k,\alpha}$ embeds continuously into $C^{k,\alpha}_{(c+2,\nu,\ldots,\nu)}(M, \Sigma^2(M))$, for any $\nu > 0$.

\
 
Thus given $\qhat \in C_c^{k,\alpha}(U,\Sigma^2(U))$ of compact support in $U$, there exists perturbations of the hyperbolic metric $S(\hhat+\qhat)$ and $T(\hhat+\qhat)$ lying in $C^{k-n+2,\alpha}_{(0,\nu,\ldots,\nu)}(M,\Sigma^2(M))$ and $C^{k,\alpha}_{(0,\nu,\ldots,\nu)}(M,\Sigma^2(M))$ respectively such that 
\begin{equation} \label{eqn:7.1}
Q(S(\hhat+\qhat),T(\hhat+\qhat)) \in C^{k-n,\alpha}_{(n-1,\nu,\ldots,\nu)}(M,\Sigma^2(M)),
\end{equation}
and moreover these maps are smooth in a neighbourhood of $\ghat$.

\

Using this asymptotic solution, we can rephrase our main theorem as follows.  

\begin{theorem}
For $n \geq 4$, and $k > n$, let $(M^n,h)$ be a geometrically finite hyperbolic metric with $n_c$ intermediate rank cusps.  If $n=4$, suppose furthermore that all the cusps are of rank $1$.  Let $U$ be any open set in $H_0$ with closure contained in $H_0\setminus \pa H_0$.  Let $\mu=(\mu_0,\mu_1,\ldots,\mu_{n_c})$ be the multi-weight given by 
$$
        \mu_0=n-2, \mu_i=\frac{1}{n-2}
$$
if $n>4$, and otherwise be a multi-weight as specified by Proposition~\ref{prop:admissible-weight2} if $n=4$.
Then there exists $\vep > 0$ such that for any smooth symmetric $2$-tensor $\qhat$ compactly supported in  $U$ and perturbation $\hat{g} = \hat{h} + \qhat$ with $\|\qhat\|_{k,\alpha} < \vep$, there is an asymptotically hyperbolic metric $g=S(\hat{h}+\qhat)+r$ on $M$ with $r\in C^{k-n,\alpha}_{\mu}(M;\Sigma^2(M))$ such that 
$\rho^2 g$ is continuous on $U$ and $(\rho^2 g)|_{U} = \hat{g}$ and  $g$ is Einstein, i.e. 
\[ \Rc(g) + (n-1)g = 0. \]
\label{mt.1}\end{theorem}

Before starting the proof, we discuss the strategy.  Given a boundary metric $\qhat$ and the approximate solutions described above, we will use the inverse function theorem in order to add a ``correction" tensor $r$ that yields an exact solution to the gauge-adjusted Einstein equation.  In order to apply the isomorphism results of Propositions \ref{prop:admissible-weight} and  \ref{prop:admissible-weight2}, we require the correction term to lie in H\"older spaces that embed into $L^2$.  This means the weight in the boundary defining function $\rho$ for $H_0$ must be at least $(n-1)/2$.  It is possible to construct an approximate solution with this weight in $\rho$ using roughly $C^{n/2}$ control of the boundary norm of $\qhat$, and we leave precise details to the interested reader.  Instead, for simplicity, we fix the weight at $H_0$ to be $\mu_0 = n-2$, which requires $C^k$ control of $\ghat$ with $k > n$ as described in the asymptotic solution above.
 
\

The remainder of this section is occupied with the proof of the main theorem.

\subsection{The inverse function theorem argument}

For some $\mu$ to be specified, define an open subset $\sB$ of the Banach space $\bR \times  C^{k-n+2,\alpha}_{\mu}(M; \Sigma^2(M))$ by
\begin{align*}
\sB &:= \{ (t,r): \hhat+t\qhat, S(\hhat+t\qhat),S(\hhat+t\qhat)+r \; \mbox{are positive definite} \}
\end{align*}
and consider the map
\begin{align*}
\sQ: \sB & \longrightarrow \bR \times  C^{k-n,\alpha}_{\mu}(M; \Sigma^2(M)), \\
 \sQ( t, r ) &= ( \; t, Q( \; S( \;\hhat+t\qhat \;) + r,  T(\; \hhat+t\qhat \;) \;) \; ). 
\end{align*}

By equation \eqref{eqn:7.1}, the corresponding discussion in \cite{GrahamLee} adapted to our notation and Corollary \ref{cor:Qmapping}, this is a smooth map with the property that $\sQ(0,0) = (0,0)$ and the linearization at this point is
\begin{align*}
 D\sQ_{(0,0)} ( \tau, r ) &= ( \tau,D_1 Q_{(h,h)} ( \tau DS_{\hat{h}} \hat{q} + r) + D_2 Q_{(h,h)}( \tau DT_{\hat{h}} \hat{q}) ) \\
 &= (\hat{p}, L r + \tau J \hat{q} ),
 \end{align*}
with $L$ defined in equation \eqref{linearize-scrQ} and $J \hat{q} =  D_1 Q_{(h,h)} DS_{\hat{h}} \hat{q}  + D_2 Q_{(h,h)}DT_{\hat{h}} \hat{q}$, exactly as in Graham-Lee.  Thus a unique solution to 
\[  D\sQ_{(\hat{h},0)} (\hat{q},r) = (\hat{w},v) \]
is given by $\hat{q} = \hat{w}$ and $r = L^{-1}(v-J\hat{w})$, provided we can invert $L$.

\

Using Propositions \ref{prop:admissible-weight} and  \ref{prop:admissible-weight2}  and our choice of  multi-weight $\mu$, our isomorphism theorem provides a bounded inverse for $L$ between weighted H\"older spaces.  Thus if $\hat{g}$ is sufficiently close to $\hat{h}$, i.e. if $\qhat$ is sufficiently small in $(k-n,\alpha)$-norm and $|t|<1$, there is a family of metrics $g_t$ for $t\in (-1,1)$ of the form $g_t = S(\hhat+t\qhat) + r_t$, with $t\mapsto r_t \in C^{k-n+2,\alpha}_{\mu}(M;\Sigma^2(M))$, such that $Q( g_t, T(\hhat+t\qhat) ) = 0$, by the inverse function theorem.  Note that by construction $\rho^2 g_t|_U = \hhat+ t\qhat$ since the tensor $\rho^2 r_t \in \rho^2 C^{k-n+2,\alpha}_{\mu}(M;\Sigma^2(M))$ vanishes on $H_0$, and the choice of weight vector ensures that the asymptotic structure of the metric is preserved.

\

Note that the metric $g$ is smooth in $M$ by interior elliptic regularity.

\

\subsection{Returning to the Einstein equation}

The final step in the proof is to argue that $g_t = S(\hhat+t\qhat) + r_t$ is an Einstein metric.  By shrinking the neighbourhood obtained from the inverse function theorem, we can always ensure that $g$ can be made as close to the hyperbolic metric $h$ as desired in $C^{k-n,\alpha}_{\mu}$-norm, and as a consequence $\Rc(g) < 0$.

\

Set $\tau = T(\hhat+t\qhat)$.  Recalling that $Q(g_t,\tau) = 0$ and that the leading part of $S(\hhat+t\qhat)$ is given by $T(\hhat+t\qhat)$, we find that $g_t = \tau + r_t'$, for $t\mapsto r_t' \in C^{k-n+2,\alpha}_{\delta}(M;\Sigma^2(M))$, where $\delta = (\mu_0 + 1, \mu_1, \ldots, \mu_{n_c} )$, in other words, $r_t'$ has faster decay at the $H_0$ hypersurface.

\

Let $\omega_t$ be the gauge-breaking DeTurck vector field
\[ \omega_t = g_t \tau^{-1} \delta_{g_t} G_{g_t} \tau. \]
Writing $\tau = g_t - r_t'$, and using the mapping properties of these operators shows that $\omega_t \in C^{k-1,\alpha}_{\delta}(M;\Sigma^2(M))$.  Thus $|\omega|^2_g = O(\sigma^{2 \delta})$, i.e. $\omega$ vanishes in each end of $M$.  Applying the Bianchi operator $\delta_{g_t} G_{g_t}$ to $Q(g_t,\tau) = 0$ and commuting derivatives shows that $|\omega_t|^2_g$ satisfies a differential inequality in $M$:
\[ \Delta |\omega_t|^2_g \leq 2 \eta |\omega_t|_g^2, \]
where $\eta$ is a negative constant that comes from the hypothesis on Ricci curvature.  As $\omega_t \to 0$ in each end, it follows that $\omega_t = 0$ by the classical maximum principle.  Thus
\[ 0 = Q(g_t,\tau) = \Rc(g_t) + (n-1)g_t + \delta_{g_t}^* \omega_t = \Rc(g_t) + (n-1) g_t, \]
and $g_t$ is Einstein.

\

This concludes the proof of Theorem \ref{mt.1}.
\bibliographystyle{amsalpha}
\bibliography{GFPE}
\end{document}